\documentclass[10pt,leqno]{amsart}
   \usepackage[centertags]{amsmath}
   \usepackage{amsfonts}
   \usepackage{amsmath}
   \usepackage{amssymb}
   \usepackage{amsthm}
   \usepackage{newlfont}
   \usepackage[latin1]{inputenc}%%%%para que reconozca las tildes.
 \usepackage{multirow, bigdelim}

\allowdisplaybreaks[3]
\theoremstyle{plain}
\newtheorem{theorem}{Theorem}[section]
\newtheorem{lemma}[theorem]{Lemma}

\newtheorem{corollary}[theorem]{Corollary}

\theoremstyle{definition}

\theoremstyle{remark}

\newtheorem*{remark*}{Remark}

\numberwithin{equation}{section}

\newcommand\D{{\mathcal D}}
\newcommand\A{{\mathcal A}}

\newcommand\F{{\mathcal F}}
\newcommand\U{{\mathcal U}}
\newcommand\V{{\mathcal V}}
\newcommand\I{{\mathcal I}}

\newcommand\CC{{\mathbb C}}
\newcommand\QQ{{\mathbb Q}}
\newcommand\RR{{\mathbb R}}
\newcommand\ZZ{{\mathbb Z}}
\newcommand\NN{{\mathbb N}}
\newcommand\PP{{\mathbb P}}
\newcommand\II{{\mathbb I}}
\newcommand\UU{{\mathbb U}}

\newcommand{\Cas}{\mathrm {Cas\,}}
\newcommand\p{\mbox{$\mathfrak{p}$}}

\newcommand\rank{\operatorname{rank}}

\newcommand\Sh{\mbox{\Large $\mathfrak {s}$}}

\numberwithin{equation}{section}

\newcommand{\dosfilas}[2]{
  \ldelim[{2}{2mm}& #1 &\rdelim]{2}{2mm} \\
  & #2 & &  & &
}

\newcommand*\pFqskip{8mu}
\catcode`,\active
\newcommand*\pFq{\begingroup
        \catcode`\,\active
        \def ,{\mskip\pFqskip\relax}%
        \dopFq
}
\catcode`\,12
\def\dopFq#1#2#3#4#5{%
        {}_{#1}F_{#2}\biggl(\genfrac..{0pt}{}{#3}{#4};#5\biggr)%
        \endgroup
}

\title[]
{Constructing bispectral dual Hahn polynomials}

\author{Antonio J. Dur\'an}
\address{A. J. Dur\'an\\
Departamento de An\'{a}lisis Matem\'{a}tico \\
Universidad de Sevilla \\
Apdo (P. O. BOX) 1160\\
41080 Sevilla. Spain.}
\email{duran@us.es}
\thanks{Partially supported by MTM2012-36732-C03-03 (Ministerio de Economía y Competitividad),
FQM-262, FQM-4643, FQM-7276 (Junta de Andalucía) and Feder Funds (European
Union).}



   \date{}
   
\begin{document}
   \maketitle

   \begin{abstract}
Using the concept of $\mathcal{D}$-operator and the classical discrete family of dual Hahn, we construct orthogonal polynomials $(q_n)_n$
which are also eigenfunctions of higher order difference operators.
\end{abstract}

\section{Introduction}
The most important families of orthogonal polynomials are the classical, classical discrete or $q$-classical families (Askey scheme and its $q$-analogue). Besides the orthogonality, they are also eigenfunctions of a second order differential, difference or $q$-difference operator, respectively. That is the case of the dual Hahn polynomials, which are eigenfunctions of a second order difference operator acting in a quadratic lattice.

The  issue of orthogonal polynomials which are also eigenfunctions of a higher order differential operator was raised by H.L. Krall in 1939, when he obtained a complete classification for the case of a differential operator of order four (\cite{Kr2}). After his pioneer work, orthogonal polynomials which are also eigenfunctions of higher order differential operators are usually called Krall polynomials. This terminology can be extended for finite order difference and $q$-difference operators. Krall polynomials are also called bispectral, following the terminology introduced by Duistermaat and Gr\"unbaum (\cite{DG}; see also \cite{GrH1}, \cite{GrH3}).

Regarding Krall polynomials, there are important differences depending whether one considers differential or difference operators.
Indeed, roughly speaking, one can construct Krall polynomials $q_n(x)$, $n\ge 0$,  by using the Laguerre $x^\alpha e^{-x}$, or Jacobi weights $(1-x)^\alpha(1+x)^\beta$, assuming that one or two of the parameters $\alpha$ and $\beta$ are nonnegative integers and adding a linear combination of Dirac deltas and their derivatives at the endpoints of the orthogonality interval (see, for instance, \cite{Kr2}, \cite{koekoe}, \cite{koe}, \cite{koekoe2}, \cite{L1}, \cite{L2}, \cite{GrH1}, \cite{GrHH}, \cite{GrY}, \cite{Plamen1}, \cite{Plamen2}, \cite{Zh}). This procedure of adding deltas seems not to work if we want to construct Krall discrete polynomials from the classical discrete measures of Charlier, Meixner, Krawtchouk and Hahn (see the papers \cite{BH} and \cite{BK} by Bavinck, van Haeringen and Koekoek answering, in the negative,
a question posed by R. Askey in 1991 (see p. 418 of \cite{BGR})).

As it has been shown recently by the author, instead of adding deltas, Krall discrete polynomials can be constructed by multiplying the classical discrete weights by certain polynomials (see \cite{du0}, \cite{ddI0}).
The kind of transformation which consists in multiplying a measure $\mu$ by a polynomial $r$ is called a Christoffel transform.
It has a long tradition in the context of orthogonal polynomials: it goes back a century and a half ago when
E.B. Christoffel (see \cite{Chr} and also \cite{Sz}) studied it for the particular case $r(x)=x$.

The purpose of this paper is to show that this procedure also works for constructing Krall polynomials from the dual Hahn orthogonal polynomials. Dual-Hahn polynomials $(R_n^{\alpha,\beta,N})_n$ are eigenfunctions of a second order difference operator acting in the quadratic lattice $\lambda(x)=x(x+\alpha+\beta+1)$.

The examples of bispectral dual Hahn polynomials constructed in this paper are also interesting by the following reason.
As it has been shown in \cite{due0} and \cite{due1}, when one applies
duality (in the sense of \cite{Leo}) to Krall-Charlier, Krall-Meixner or Krall-Krawtchouk orthogonal polynomials, exceptional discrete polynomials appear. Exceptional and exceptional discrete orthogonal polynomials $p_n$, $n\in X\varsubsetneq \NN$, are complete orthogonal polynomial systems with respect to a positive measure which in addition are eigenfunctions of a second order differential or difference operator, respectively. They extend the  classical families of Hermite, Laguerre and Jacobi or the classical discrete families of Charlier, Meixner and Hahn.

The last few years have seen a great deal of activity in the area  of exceptional orthogonal polynomials (see, for instance,
\cite{due0}, \cite{due1},\cite{GUKM2} (where the adjective \textrm{exceptional} for this topic was introduced),  \cite{GUGM}, \cite{GQ}, \cite{MR},  \cite{OS3}, \cite{OS4}, \cite{STZ}, and the references therein).
 The most apparent difference between classical or classical discrete orthogonal polynomials and their exceptional counterparts
is that the exceptional families have gaps in their degrees, in the
sense that not all degrees are present in the sequence of polynomials (as it happens with the classical families) although they form a complete orthonormal set of the underlying $L^2$ space defined by the orthogonalizing positive measure. This
means in particular that they are not covered by the hypotheses of Bochner's and Lancaster's classification theorems for classical and classical discrete orthogonal polynomials, respectively (see \cite{B} or \cite{La}).

The connection by duality between Krall discrete and exceptional discrete polynomials is remarkable because anything similar it is known to happen for classical polynomials and differential operators. For exceptional Charlier and Meixner polynomials one can construct exceptional Hermite and Laguerre polynomials by passing to the limit in the same form as one goes from Charlier and Meixner to Hermite and Laguerre in the Askey tableau. The relation between Krall and exceptional polynomials at the level of the classical discrete families is rather helpful even at the classical level. For instance, using it one can lighten the difficult problem of finding necessary and sufficient conditions for the existence of a positive measure with respect to which the exceptional Hermite and Laguerre polynomials are orthogonal. In a forthcoming paper \cite{due2}, we will construct Hahn polynomials by applying duality to the Krall dual Hahn polynomials constructed in this paper, with some applications to the construction of exceptional Jacobi polynomials.

For the dual Hahn polynomials $(R_n^{\alpha,\beta, N})_n$, we have to work in the subring $\PP^\lambda $
of $\PP$ consisting of polynomials in the variable $\lambda$, where $\lambda(x)=x(x+\alpha+\beta+1)$:
\begin{equation}\label{defsrdh}
\PP^\lambda=\{p(\lambda(x)):p\in \PP\}.
\end{equation}
The measure with respect to which our bispectral dual Hahn polynomials are orthogonal, will be defined by applying a Christoffel transform to the dual Hahn measure. More precisely, for real numbers $\alpha,\beta$  and a positive integer $N$, we denote by $\rho _{\alpha,\beta,N}$ the dual Hahn weight (see (\ref{masdh}) below). Let $\F=(F_1,F_2,F_3)$  be a trio of finite sets of positive integers (the empty set is allowed). Under mild condition on the parameters $\alpha,\beta$ and $N$, we will then prove (in a constructive way) that the weight $\rho _{\alpha,\beta,N}^{\F}$ defined by
\begin{equation}\label{mii}
\rho _{\alpha,\beta,N}^{\F}=\prod_{f\in F_1}(\lambda-\lambda(N-f))\prod_{f\in F_2}(\lambda-\lambda(f))\prod_{f\in F_3}(\lambda-\lambda(f-\beta))\rho _{\alpha,\beta,N},
\end{equation}
has associated a sequence of orthogonal polynomials and they are eigenfunctions of a higher order difference operator (for an example of orthogonal polynomials constructed from the dual Hahn family and satisfying fourth order difference equations see \cite{TVZ}).

In order to prove this result, we use $\D$-operators and the approach developed in \cite{ddI0} for constructing Krall polynomials from the Charlier, Meixner and Krawtchouk families.

This approach has three  ingredients (which will be considered in Section 3).

\medskip
\noindent
\textbf{The first ingredient: $\D$-operators}. This is an abstract concept introduced in \cite{du1} by the author which has shown to be very useful to generate Krall, Krall discrete and $q$-Krall families of polynomials (see \cite{du1}, \cite{AD}, \cite{ddI0}, \cite{ddI1}).

To define a $\D$-operator, we need a sequence of polynomials (in $\lambda$) $(p_n)_n$, $\deg_\lambda p_n=n$, and an algebra of operators $\A $ acting in $\mathbb{P}^\lambda$. In addition, we assume that the polynomials $p_n$, $n\ge 0$, are eigenfunctions of certain operator $D_p\in \A$ with eigenvalues that are linear in $n$: that is, we assume that $D_p(p_n)=np_n$, $n\ge 0$. Observe that no orthogonality conditions are imposed at this stage on the polynomials $(p_n)_n$.
Given  a sequence of numbers $(\varepsilon_n)_n$, a $\D$-operator $\D$ associated to the algebra $\A$ and the sequence of polynomials
$(p_n)_n$ is defined  by linearity in $\PP$ from
\begin{equation}\label{Dhi}
\D(p_n)=\sum _{j=1}^n (-1)^{j+1}\varepsilon_n\cdots \varepsilon_{n-j+1}p_{n-j},\quad n\ge 0.
\end{equation}
We then say that the lowering operator $\D$ is a $\D$-operator if $\D\in \A$.

Using $\D$-operators we can construct from the polynomials $(p_n)_n$ a huge class of families of polynomials $(q_n)_n$ which are also eigenfunctions of operators in the algebra $\A$. Indeed, assume we have $m$ $\D$-operators $\D_1, \D_2, \ldots, \D_m$ (not necessarily different) defined by the sequences $(\varepsilon _n^h)_n$, $h=1,\ldots , m$, and that these sequences are also defined for $n\in \ZZ$.
We write $\xi_{x,i}^h$, $i\in\ZZ$ and $h=1,2,\ldots,m$, for the auxiliary functions defined by
\begin{equation*}\label{defxiii}
\xi_{x,i}^h=\prod_{j=0}^{i-1}\varepsilon_{x-j}^{h}, \quad i\ge 1,\quad \quad \xi_{x,0}^h=1,\quad\quad \xi_{x,i}^h=\frac{1}{\xi_{x-i,-i}^h},\quad i\leq-1.
\end{equation*}

For $m$ arbitrary polynomials $Y_1, Y_2, \ldots, Y_m$, we consider the sequence of polynomials $(q_n)_n$ defined by
\begin{equation}\label{qusi}
q_n(\lambda)=\begin{vmatrix}
               p_n(\lambda) & -p_{n-1}(\lambda) & \cdots & (-1)^mp_{n-m}(\lambda) \\
               \xi_{n,m}^1Y_1(n) &  \xi_{n-1,m-1}^1Y_1(n-1) & \cdots & Y_1(n-m) \\
               \vdots & \vdots & \ddots & \vdots \\
                \xi_{n,m}^mY_m(n) &  \xi_{n-1,m-1}^mY_m(n-1) & \cdots & Y_m(n-m)
             \end{vmatrix}.
\end{equation}
To ensure that at least a finite family of the polynomials $q_n$ has degree $n$ we assume that there exists a positive integer $M$ such that the (quasi) Casorati determinant
$$
\Omega (n)=\det \left(\xi_{n-j,m-j}^lY_l(n-j)\right)_{l,j=1}^m,
$$
satisfies that $\Omega(n)\not =0$ for $n=0,1,\cdots , M$ (this is necessary if we want the polynomials $q_n$, $0\le n\le M$, to be orthogonal).

Notice that the dependence in $\lambda$ in the determinant (\ref{qusi}) appears only in the first row, and hence $q_n$ is a linear combination of $m+1$ consecutive $p_n$'s. The magic of  $\D$-operators is that, whatever the polynomials $Y_j$'s are, there always exists an operator $D_q$ in the algebra $\A$ for which the polynomials $q_n$, $0\le n\le M$, are eigenfunctions. Moreover, the operator $D_q$ can be explicitly constructed from the operator $D_p$ using the $\D$-operators $\D_j$, $j=1,\ldots , m$. To stress the dependence of the polynomials $q_n$, $n\ge 0$, on the polynomials $Y_j$, $j=1,\ldots , m$, we write $q_n=\Cas_{n}^{Y_1,\ldots , Y_m}$. Polynomials defined by  other similar forms of Casorati determinants have also a long tradition in the context of orthogonal polynomials and bispectral polynomials. Casorati determinants appear, for instance, to express orthogonal polynomials with respect to the Christoffel or Geronimus transform of a measure. See  \cite{Sz}, Th. 2.5 for the Christoffel transform and \cite{Gs}, \cite{Gs2} or \cite{Zh0} (and the references therein) for the Geronimus transform. The Geronimus transform associated to the polynomial $q(x)=(x-f_1)\cdots (x-f_k)$ is defined as follows: we say that $\tilde \mu$ is a Geronimus transform of $\mu$ if $q\tilde \mu =\mu$. Notice that the Geronimus transform is reciprocal of the Christoffel transform (see the preliminaries for more details).

In this paper, we construct three different $\D$-operators for dual Hahn polynomials (see Lemma \ref{dhdo} in Section \ref{sch}). One important difference between dual Hahn polynomials and the Charlier, Meixner or Krawtchouk families consider in \cite{ddI0} is that for dual Hahn polynomials one of the sequences $\varepsilon _n$ vanishes at certain positive integer.
For the benefit of the reader, we display here this sequence and the associated $\D$-operator: the sequence given by
$$
\varepsilon_{n}=\frac{N-n+1}{\alpha +n},
$$
defines the following $\mathcal{D}$-operator (see \eqref{Dhi}) for the dual Hahn polynomials:
$$
\mathcal{D}=\frac{(x+\alpha+\beta+1)(N-x)}{(2x+\alpha+\beta+1)(2x+\alpha+\beta+2)}\Delta_x+
\frac{x(x+\alpha+\beta+N+1)}{(2x+\alpha+\beta)(2x+\alpha+\beta+1)}\nabla_x.
$$
Notice that $\varepsilon_{N+1}=0$.

\bigskip

\noindent
\textbf{The second ingredient.} With the second ingredient, orthogonality with respect to a measure enters into the picture (Section 4). Indeed, even if we assume that the polynomials $(p_n)_n$ are orthogonal, only for a convenient choice of the polynomials $Y_j$, $j=1,\ldots, m$, the polynomials (\ref{qusi}) $q_n=\Cas_n^{Y_{1},\ldots , Y_{m}}$, $n\ge 0$, are also orthogonal with respect to a measure. When we take the polynomials $(p_n)_n$ to be the dual Hahn polynomials, the second ingredient establishes how to chose the polynomials $Y_j$'s such that the polynomials $q_n=\Cas_n^{Y_1,\ldots ,Y_m}$ (\ref{qusi}) are also orthogonal with respect to a measure. As for the case of Charlier, Meixner and Krawtchouk (studied in \cite{ddI0}), this second ingredient turns into a very nice symmetry between the dual Hahn family and the polynomials $Y_j$'s. Indeed, the polynomials $Y_j$'s can be chosen to be Hahn polynomials, but with a suitable modification of the parameters. More precisely, given a $\D$-operator for the dual Hahn family and a nonnegative integer $j$ we provide a polynomial $Y_j$ of degree $j$ such that for any different nonnegative integers $g_1,\ldots , g_m$, the polynomials $q_n=\Cas_n^{Y_{g_1},\ldots , Y_{g_m}}$, $n\ge 0$, are orthogonal with respect to a measure $\tilde \rho$. For the $\D$-operator display above we have $Y_j(x)=h_j^{-\alpha,\beta,-\beta-2-N}(-x-1)$, where by $(h_n^{\alpha,\beta,N})_n$ we denote the Hahn polynomials (see (\ref{hxpol}) below).
\bigskip

\noindent
\textbf{The third ingredient.} We still need a last ingredient for identifying the measure $\tilde \rho$ with respect to which the polynomials $q_n=\Cas_n^{Y_{g_1},\ldots ,Y_{g_m}}$ (\ref{qusi}) are orthogonal. These polynomials depend on the set of indices $G=\{g_1,\ldots , g_m\}$ (the degrees of the polynomials $Y_{g_j}$). It turns out that the orthogonalizing measure for the polynomials $(\Cas_n^{Y_{g_1},\ldots ,Y_{g_m}})_n$ is one of the measures (\ref{mii}) introduced above. These measures depend on certain finite sets $F$'s of positive integers. The third ingredient establishes the relationship between these sets $F$'s and the set $G$. This relationship is given by suitable transforms defined in the set $\Upsilon$ formed by all finite sets of positive integers.
\medskip

In Section 5 we will put together all these ingredients to construct bispectral dual Hahn polynomials.

We finish pointing out that, as explained above, the approach of this paper is the same as in \cite{ddI0} for Charlier, Meixner and Krawtchouk  polynomials. Since we work here in a quadratic lattice with a trio of finite sets of positive integers (instead of at most two sets as in \cite{ddI0}), and more parameters, the computations are technically more involved. Anyway, we will omit those proofs which are too similar to the corresponding ones in \cite{ddI0}.

\section{Preliminaries}
For a linear operator $D:\PP \to \PP$ and a polynomial $P(x)=\sum _{j=0}^ka_jx^j$, the operator $P(D)$ is defined in the usual way
$P(D)=\sum _{j=0}^ka_jD^j$.

For a moment functional $\mu$ on the real line, that is, a linear mapping $\mu :\PP \to \RR$, the $n$-th moment of $\mu $ is defined by $\mu_n=\langle \mu, x^n\rangle $.
It is well-known that any moment functional on the real line can be represented by integrating with respect to a Borel measure
(positive or not) on the real line
(this representation is not unique \cite{du-1}).
If we also denote this measure by $\mu$, we have $\langle \mu,p\rangle=\int p(x)d\mu(x)$ for all polynomial $p\in \PP$. Taking this into account,
we will conveniently use along this paper one or other terminology (orthogonality with respect to a moment functional or with
respect to a measure).
We say that a sequence of polynomials  $(p_n)_n$, $p_n$ of degree $n$, $n\ge 0$, is orthogonal with respect to the moment functional $\mu$ if $\langle \mu, p_np_m\rangle=0$, for $n\not =m$ and $\langle \mu, p_n^2\rangle\not =0$. Since the most important examples consider in this paper are orthogonal polynomials with respect to a degenerate measure (when $N$ is a positive integer the dual Hahn polynomials $R_n^{\alpha,\beta,N}$ are orthogonal with respect to a finite combination of deltas), we will stress this property of non-vanishing norm when necessary.

As we wrote in the introduction, for the dual Hahn polynomials $(R_n^{\alpha,\beta, N})_n$, we have to work in the subring $\PP^\lambda $
defined by (\ref{defsrdh}). Hence for a moment functional $\mu$ on the real line, we also denote by $\mu$ the corresponding moment functional in $\PP^\lambda$ defined by $\langle \mu, p(\lambda)\rangle=\langle \mu, p(\lambda(x))\rangle$.

Favard's Theorem establishes that a sequence $(p_n)_{0\le n\le M}$ (where $M$ is a positive integer or infinity) of polynomials, $p_n$ of degree $n$, is orthogonal (with non null norm) with respect to a measure if and only if it satisfies a three term recurrence relation of the form ($p_{-1}=0$)
\begin{equation}\label{fvo}
\lambda p_n(\lambda )=a_{n+1}p_{n+1}(\lambda )+b_np_n(\lambda )+c_np_{n-1}(\lambda ), \quad 0\le n\le M-1,
\end{equation}
where $(a_n)_{n}$, $(b_n)_{n}$ and $(c_n)_{n}$ are sequences of real numbers with $a_{n}c_n\not =0$, $1\le n\le M$. If, in addition, $a_{n}c_n>0$, $1\le n\le M$,
then the polynomials $(p_n)_{0\le n\le M}$ are orthogonal with respect to a positive measure, and the reciprocal is also true. If $M=\infty$, this measure will have infinitely many points in its support, otherwise the support might be formed by finitely many points.

\bigskip

As we wrote in the Introduction, the kind of transformation which consists in multiplying a moment functional $\mu$ by a polynomial $r$ is called a Christoffel transform. The new moment functional $r\mu$ is defined by $\langle r\mu,p\rangle =\langle \mu,rp\rangle $. Its reciprocal is the Geronimus transform $\tilde \mu$ which satisfies $r\tilde \mu =\mu$. Notice that the Geronimus transform of the moment functional $\mu$ is not uniquely defined. Indeed, write $a_i$, $i=1,\ldots , u$, for the different real roots of the polynomial $r$, each one with multiplicity $b_i$, respectively.
It is easy to see that if $\tilde \mu$ is a Geronimus transform of $\mu$ then the moment functional $\tilde \mu +\sum_{i=1}^u\sum_{j=0}^{b_i-1}\alpha_{i,j}\delta _{a_i}^{(j)}$ is also a Geronimus transform of $\mu$, where $\alpha_{i,j}$ are real numbers. These numbers are usually called the free parameters of the Geronimus transform.

In the literature, Geronimus transform is sometimes called Darboux transform with parameters while Christoffel transform is called Darboux transform without parameters. The reason is the following. The three term recurrence relation (\ref{fvo}) for the orthogonal polynomials with respect to $\mu$ can be rewritten as $\lambda p_n=J(p_n)$, where $J$ is the second order difference operator $J=a_{n+1}\Sh_1+b_n\Sh_0+c_n\Sh_{-1}$ and $\Sh_l$ the shift operator (acting on the discrete variable $n$): $\Sh_l(x_n)=x_{n+l}$.
For any $f \in \CC $, decompose $J$ into $J=AB+f I$
whenever it is possible, where $A=\alpha_n\Sh_0+\beta_n\Sh_1$ and $B= \delta_n\Sh_{-1}+\gamma_n\Sh_0$. We then call $\tilde J=BA+f I$
a Darboux transform of $J$ with parameter $f$. It turns out that the second order difference operator $\tilde J$ associated to a Geronimus transform $\tilde \mu$ of $\mu$ can be obtained by applying a sequence of  $k$ Darboux transforms (with parameters $f_i$, $i=1,\ldots, k$) to the operator $J$ associated to the measure $\mu$. This kind of Darboux transform has been used by Gr\"unbaum, Haine, Hozorov, Yakimov and Iliev to construct Krall and $q$-Krall polynomials. For the particular cases of Laguerre, Jacobi or Askey-Wilson polynomials, one can found Casorati determinants similar to (\ref{qusi}) in \cite{GrHH}, \cite{GrY}, \cite{HP}, \cite{Plamen1} or \cite{Plamen2}.

The family of measures $\rho _{\alpha,\beta,N}^\F$ in the Introduction is defined by applying a Christoffel transform to the dual Hahn weight. But, it turns out that they can also be defined by using the Geronimus transform. This Geronimus transform is however defined by a different polynomial. We also have to make
a suitable choice of the free parameters of this Geronimus transform and apply it to a dual Hahn weight but maybe with different parameters and affected by a shift in the variable. The following example will clarify this point. Consider $F_1=\{1\}$, $F_2=\{ 1\}$, $F_3=\{1\}$ and the Christoffel transform $\rho_{\alpha,\beta,N}^{\F}$ of the dual Hahn weight $\rho _{\alpha,\beta,N}$ defined by the polynomial $p(\lambda)=(\lambda-\lambda(1))(\lambda-\lambda(N-1))(\lambda-\lambda(1-\beta))$, where $\lambda(x)=x(x+\alpha+\beta+1)$. That is
$$
\rho_{\alpha,\beta,N}^{\F}=(\lambda-\lambda(1))(\lambda-\lambda(N-1))(\lambda-\lambda(1-\beta))\rho _{\alpha,\beta,N}.
$$
From the definition of the dual Hahn weight $\rho _{\alpha,\beta,N}$, we have after a simple computation
\begin{align*}
\rho_{\alpha,\beta,N}^{\F}&=p(0)w_{*;\alpha,\beta,N}(0)\delta _0+p(\lambda(N))w_{*;\alpha,\beta,N}(N)\delta_N\\&\hspace{1cm}+(\alpha +1)_4(N-3)_4^2\sum_{x=2}^{N-2} \frac{w_{*;\alpha+4,\beta,N-4}(x-2)}{\lambda(x)(\lambda(x)-\lambda(N))(\lambda(x)-\lambda(-1-\alpha))}\delta _x,
\end{align*}
where $w_{*;\alpha,\beta,N}(x)$ denotes the mass at $x$ of $\rho _{\alpha,\beta,N}$ (see (\ref{masdh}) below).
This shows that
$$
\lambda(\lambda-\lambda(N))(\lambda-\lambda(-1-\alpha))\rho_{\alpha,\beta,N}^{\F}=(\alpha +1)_4(N-3)_4^2\rho_{\alpha+4,\beta,N-4}(x-2).
$$
That is, $\rho_{\alpha,\beta,N}^{\F}$ is also the Geronimus transform defined by the polynomial $\lambda(\lambda-\lambda(N))(\lambda-\lambda(-1-\alpha))$ of the dual Hahn weight $(\alpha +1)_4(N-3)_4^2\rho_{\alpha+4,\beta,N-4}(x-2)$ where the free parameters (associated to the roots $\lambda=0,\lambda=\lambda(N)$ and $\lambda=\lambda(-1-\alpha)$) have to be necessarily chosen equal to $p(0)w_{*;\alpha,\beta,N}(0)$, $p(\lambda(N))w_{*;\alpha,\beta,N}(N)$ and $0$, respectively.

\bigskip

Along this paper, we use the following notation:
given a finite set of positive integers $F=\{f_1,\ldots , f_m\}$, the expression
\begin{equation}\label{defdosf}
  \begin{array}{@{}c@{}cccc@{}c@{}}
    &&&\hspace{-.9cm}{}_{1\le j\le m} \\
    \dosfilas{ z_{f,j}  }{f\in F}
  \end{array}
\end{equation}
inside of a matrix or a determinant will mean the submatrix defined by
$$
\left(
\begin{array}{cccc}
z_{f_1,1} & z_{f_1,2} &\cdots  & z_{f_1,m}\\
\vdots &\vdots &\ddots &\vdots \\
z_{f_m,1} & z_{f_m,2} &\cdots  & z_{f_m,m}
\end{array}
\right) .
$$

\section{The main ingredients}
\subsection{$\D$-operators}
The concept of $\D$-operator was introduced by the author in the paper \cite{du1}. In  \cite{du1}, \cite{ddI0}, \cite{ddI1} and \cite{AD}, it has been showed that $\D$-operators turn out to be an extremely useful tool of a unified method
to generate families of polynomials which are eigenfunctions of higher order differential, difference or $q$-difference operators.
Hence, we start by reminding the concept of $\D$-operator.

As we wrote in the Introduction, for the dual Hahn polynomials $(R_n^{\alpha,\beta, N})_n$, we have to work in the subring $\PP^\lambda $
of $\PP$ consisting of polynomials in the variable $\lambda$, where $\lambda(x)=x(x+\alpha+\beta+1)$ (see (\ref{defsrdh})).

The subring $\PP^\lambda$ can be easily characterized as follows. Consider the involution $\I:\PP\to \PP$ defined by
\begin{equation}\label{definvsr}
\I(p)=p(-(x+\alpha+\beta+1)).
\end{equation}
Clearly we have $\I(\lambda) = \lambda$. Hence every polynomial in $\PP^\lambda$ is invariant under the
action of $\I$. And conversely, if $p\in \PP$ is invariant under $\I$, then $p\in \PP^\lambda$.
We consider the shift operators in $\PP^\lambda$ acting on $x$:
\begin{equation*}
\Sh_j(p(\lambda))=p(\lambda(x+j)).
\end{equation*}
To stress that we will sometimes write $\Sh_{x,j}$ instead of $\Sh_j$. Notice that for $p\in \PP^\lambda$,
$\Sh_{x,j}(p)$ does not belong, in general, to $\PP^\lambda$.

We consider difference operators $T$ in $\PP^\lambda$ of the form
\begin{equation}\label{doho}
T=\sum_{j=s}^rh_j\Sh_{x,j}^\lambda
\end{equation}
with $h_j\in \QQ[x]$, $j=s,\ldots,r, s\leq r$, and where $\QQ[x]$ denotes the linear space of rational functions.
We denote by $\A^\lambda$ the algebra formed by all the operators $T$ of the form (\ref{doho}) which maps $\PP^\lambda$ into itself:
\begin{equation}\label{defal}
\A^\lambda=\{T: \mbox{$T$ is of the form (\ref{doho}) and $T(\PP^\lambda)\subset \PP^\lambda$.}\}
\end{equation}
The second order difference operator for the dual Hahn polynomials (\ref{defbc}) belongs then to $\A^\lambda$.

The starting point to define a $\D$-operator is a sequence of polynomials (in $\lambda$) $(p_n)_n$, $\deg_\lambda p_n=n$, and a subalgebra of operators $\A $ of the algebra $\A^\lambda$ (hence acting in the subring $\mathbb{P}^\lambda$ and mapping it into itself).

In addition, we assume that the polynomials $p_n$, $n\ge 0$, are eigenfunctions of certain operator $D_p\in \A$. We write $(\theta_n)_n$ for the corresponding eigenvalues, so that $D_p(p_n)=\theta_np_n$, $n\ge 0$. Since we are interested in the dual Hahn polynomials, we only consider here the case when the sequence of eigenvalues $(\theta _n)_n$ is linear in $n$.

Given  a sequence of numbers $(\varepsilon_n)_n$, a $\D$-operator associated to the algebra $\A$ and the sequence of polynomials
$(p_n)_n$ is defined as follows.
We first consider  the operator $\D :\PP \to \PP $ defined by linearity
from
\begin{equation}\label{defTo}
\D (p_n)=\sum _{j=1}^n (-1)^{j+1}\varepsilon _n\cdots \varepsilon _{n-j+1}p_{n-j},\quad n\ge 0.
\end{equation}
We then say that the lowering operator $\D$ is a $\D$-operator if $\D\in \A$.

The following Theorem was proved in \cite{ddI0} and shows how to use $\D$-operators to construct new sequences of polynomials $(q_n)_n$ such that there exists an operator $D_q\in \A$ for which they are eigenfunctions. We use $m$ arbitrary polynomials (in $x$) $Y_1, Y_2, \ldots, Y_m$ and $m$ $\D$-operators $\D_1, \D_2, \ldots, \D_m$ (not necessarily different) defined by the sequences $(\varepsilon _n^h)_n$, $h=1,\ldots , m$:
\begin{equation}\label{Dh}
\mathcal{D}_h(p_n)=\sum _{j=1}^n (-1)^{j+1}\varepsilon_{n}^{h}\cdots\varepsilon_{n-j+1}^{h}p_{n-j},\quad h=1,2,\ldots,m.
\end{equation}

We will assume that for $h=1,2,\ldots,m$, the sequence $(\varepsilon_{n}^{h})_n$ is a rational function in $n$ (actually that is the case for the three $\D$-operators we will construct in Section \ref{sch} for the dual Hahn polynomials). We write $\xi_{x,i}^h$, $i\in\ZZ$ and $h=1,2,\ldots,m$, for the auxiliary functions defined by
\begin{equation}\label{defxi}
\xi_{x,i}^h=\prod_{j=0}^{i-1}\varepsilon_{x-j}^{h}, \quad i\ge 1,\quad \quad \xi_{x,0}^h=1,\quad\quad \xi_{x,i}^h=\frac{1}{\xi_{x-i,-i}^h},\quad i\leq-1.
\end{equation}
We will consider the $m\times m$ (quasi) Casorati determinant defined by
\begin{equation}\label{casd1}
\Omega (x)=\det \left(\xi_{x-j,m-j}^lY_l(x-j)\right)_{l,j=1}^m.
\end{equation}

The next Theorem is a slight modified version of Theorem 3.2 of \cite{ddI0} (adapted to the particularities of dual Hahn polynomials). This Theorem shows how to use $\D$-operators to construct new sequences of polynomials $(q_n)_n$  such that there exists an operator $D_q$ in $\A$ for which they are eigenfunctions

\begin{theorem}[Theorem 3.2 of \cite{ddI0}]\label{Teor1} Let $\A$ and $(p_n)_n$ be, respectively, a subalgebra of operators $\A $ of the algebra $\A^\lambda$ (hence acting in the subring $\mathbb{P}^\lambda$ and mapping it into itself), and a sequence of polynomials (in $\lambda$) $(p_n)_n$, $\deg_\lambda p_n=n$. We assume that $(p_n)_n$ are eigenfunctions of an operator $D_p\in \A$ with eigenvalues equal to $n$, that is, $D_p(p_n)=np_n$, $n\ge 0$. We also have $m$ sequences of numbers $(\varepsilon_{n}^{1} )_n,\ldots,(\varepsilon_{n}^{m} )_n$, which define $m$ $\D$-operators $\D_1,\ldots,\D_m$ (not necessarily different) for $(p_n)_n$ and $\A$ (see \eqref{Dh})) and assume that for $h=1,2,\ldots,m$, each sequence $(\varepsilon_{n}^{h})_n$ is a rational function in $n$.

Let $Y_1, Y_2, \ldots, Y_m$ be $m$ arbitrary polynomials (in $x$) satisfying that $\Omega (x)\not =0$, $x\in \NN$, $0\le x\le M$, where $M$ is certain positive integer or infinity and $\Omega $ is the Casorati determinant defined by \eqref{casd1}.

Consider the sequence of polynomials $(q_n)_n$ defined by
\begin{equation}\label{qus}
q_n(\lambda)=\begin{vmatrix}
               p_n(\lambda) & -p_{n-1}(\lambda) & \cdots & (-1)^mp_{n-m}(\lambda) \\
               \xi_{n,m}^1Y_1(n) &  \xi_{n-1,m-1}^1Y_1(n-1) & \cdots & Y_1(n-m) \\
               \vdots & \vdots & \ddots & \vdots \\
                \xi_{n,m}^mY_m(n) &  \xi_{n-1,m-1}^mY_m(n-1) & \cdots & Y_m(n-m)
             \end{vmatrix}.
\end{equation}
For a rational function $S(x)$ and $h=1,\ldots,m$, we define the function $M_h(x)$ by
\begin{equation}\label{emeiexp}
M_h(x)=\sum_{j=1}^m(-1)^{h+j}\xi_{x,m-j}^hS(x+j)
\det\left(\xi_{x+j-r,m-r}^{l}Y_l(x+j-r)\right)_{l\in \II_h;r\in \II_j},
\end{equation}
where $\II_h=\{1,2,\ldots,m\}\setminus\{h\}$.
If we assume that the functions $S(x)\Omega (x)$ and $M_h(x)$, $h=1,\ldots,m$, are polynomials in $x$
then there exists an operator $D_{q,S}\in \A$ such that
$$
D_{q,S}(q_n)=\lambda_nq_n,\quad 0\le n\le M.
$$
Moreover, an explicit expression of this operator can be displayed. Indeed, write  $P_S$ for the polynomial defined by
\begin{equation}\label{Pgs}
P_S(x)-P_S(x-1)=S(x)\Omega (x).
\end{equation}
Then the operator $D_{q,S}$ is defined by
\begin{equation}\label{Dq}
D_{q,S}=P_S(D_p)+\sum_{h=1}^mM_h(D_p)\D_hY_h(D_p),
\end{equation}
where $D_p\in \A$ is the operator for which the polynomials $(p_n)_n$ are eigenfunctions. Moreover
$\lambda_n=P_S(n)$.

\end{theorem}

Notice that the dependence in $\lambda$ of the polynomials (\ref{qus}) appears only in the first row, and hence $q_n$ is a linear combination of $m+1$ consecutive $p_n$'s.

\bigskip

In Section \ref{sch},  we will apply Theorem \ref{Teor1} to the dual Hahn polynomials. We will see there that
the degree of the polynomial $P_S$ (see \eqref{Pgs}) will give the order of the difference operator $D_{q,S}$ \eqref{Dq} with respect to which the new polynomials $(q_n)_n$ are eigenfunctions. This will be a consequence of the following Lemma (which we will prove in Section \ref{sproofs}).

\begin{lemma}\label{lgp2} With the same notation as in the previous Theorem, write
$$
\Psi_j^h(x)=\xi_{x-j,m-j}^hS(x)
\det\left(\xi_{x-r,m-r}^{l}Y_l(x-r)\right)_{l\in \II_h;r\in \II_j}, \quad h,j=1,\cdots , m,
$$
and $\Omega _g^h$, $h=1,\cdots , m$, $g\in \NN$, for the particular case of $\Omega$ when $Y_h(x)=x^g$.
Assume that $\Psi_j^h$, $h,j=1,\cdots , m$, are polynomials in $x$ and write $\tilde d=\max\{\deg \Psi_j^h:h,j=1,\cdots , m\}$.
Then $M_h$ and $S\Omega _g^h$, $h=1,\cdots , m$, $g\in \NN$, are also polynomials in $x$. If, in addition, we assume that the degree of $S(x)\Omega _g^h$ is at most $g+\deg (S(x)\Omega_0^h)$ for $h=1,\cdots, m$ and $g\le \tilde d-\deg (S(x)\Omega_0^h)$, then $M_h$ is a polynomial of degree at most $\deg (S(x)\Omega_0^h)$.
\end{lemma}

To compute the degree of the polynomial $P_S$ (see \eqref{Pgs}) we will use the following Lemma (which we will also prove in Section \ref{sproofs}).

For a complex number $u\in \CC$, write $s_j^u$, $j=0,1,2,\cdots $, for the polynomial
\begin{equation}\label{defsj}
s_j^u(x)=(u-x)_{j}.
\end{equation}

Given a trio $\U=(U_1,U_2,U_3)$  of finite sets of nonnegative integers, we write $m_j$
for the number of elements of $U_j$, $j=1,2,3$, $m=m_1+m_2+m_3$ and
\begin{equation}\label{ppmm}
\UU_1=\{1,\cdots, m_1\},\quad \UU_2=\{m_1+1,\cdots, m_1+m_2\},\quad \UU_3=\{m_1+m_2+1,\cdots, m\}.
\end{equation}
We write $U_j=\{u_i^{j\rceil},i\in \UU_j\}$.

\begin{lemma}\label{lgp1} Let $Y_1, Y_2, \ldots, Y_m,$ be nonzero polynomials satisfying that $\deg Y_i=u_{i}^{j\rceil}$, if $i\in \UU_j$ and $1\le j\le 3$. Write $r_i$ for the leading coefficient of $Y_i$, $1\le i\le m$.
For real numbers $N, \alpha, \beta$, consider the rational function $P$ defined by
\begin{equation}\label{def2p}
P(x)=\frac{\left|
  \begin{array}{@{}c@{}lccc@{}c@{}}
    &&&\hspace{-.9cm}{}_{1\le j\le m} \\
    \dosfilas{ Y_{i}(x-j) }{i\in \UU_1} \\
    \dosfilas{ s_{m-j}^{\beta+N+1}(x-j)s_{j-1}^{-\alpha+1}(x)Y_{i}(x-j) }{i\in \UU_2}\\
     \dosfilas{ s_{m-j}^{N+1}(x-j)s_{j-1}^{-\alpha+1}(x)Y_{i}(x-j) }{i\in \UU_3}
  \end{array}
  \hspace{-.4cm}\right|}{p(x)},
\end{equation}
where $p$ is the polynomial
\begin{align}\label{def1p}
p(x)&=\prod_{i=0}^{m_2-2}(\beta+N+m-x-i)^{m_2-i-1}\prod_{i=0}^{m_3-2}(N+m-x-i)^{m_3-i-1}
 \\\nonumber &\qquad \times \prod_{i=0}^{m_2+m_3-2}(-\alpha-x+i+1)^{m_2+m_3-i-1}.
\end{align}
The determinant (\ref{def2p}) should be understood in the way explained in the Preliminaries (see (\ref{defdosf})).
If
\begin{equation}\label{yas0}
\beta-v+w,\alpha+\beta +N+1+u-v,\alpha+N+1+u-w\not =0,
\end{equation}
for $u\in U_1,v\in U_2,w\in U_3$, then $P$ is a polynomial of degree
\begin{equation}\label{ddp}
d=\sum_{u\in U_1,U_2,U_3}u-\binom{m_1}{2}-\binom{m_2}{2}-\binom{m_3}{2}
\end{equation}
with leading coefficient given by
\begin{align}\label{mspcl}
r&=(-1)^{\binom{m}{2}+m_2m_3}V_{U_1}V_{U_2}V_{U_3}\prod_{i=1}^kr_i\prod_{v\in U_2,w\in U_3}(\beta -v+w)
\\\nonumber&\qquad \times\prod_{u\in U_1,v\in U_2}(\alpha+\beta +N+1+u-v)\prod_{u\in U_1,w\in U_3}(\alpha+N+1+u-w).
\end{align}
\end{lemma}

\subsection{When are the polynomials $(\Cas_n^{Y_1,\ldots ,Y_m})_n$ orthogonal?} \label{ssi}
Only for a convenient choice of the polynomials $Y_j$, $j=1,\ldots, m$, the polynomials $(q_n)_n$ (\ref{qus}) are also orthogonal with respect to a measure. In \cite{ddI0}, Sect. 4, a method to check the orthogonality of the polynomials $(q_n)_n$ (\ref{qus}) was provided.
This tool assumed that the sequences $(\varepsilon_{n}^{h})_{n\in \ZZ }$, $h=1,\ldots ,m$, do not vanishes for any $n$. This is not the case for the dual Hahn polynomials, but, as we next show, that method can be modified to include also that case.

Indeed, given the sequences (in the integers) $(a_n)_{n\in \ZZ }$, $(b_n)_{n\in \ZZ }$, $(c_n)_{n\in \ZZ }$ and
$(\varepsilon_{n}^{h})_{n\in \ZZ }$, $h=1,\ldots ,m$,  and $m$ functions $\lambda _h$, $h=1,\ldots ,m$, assume we have for each $j\ge 0$ and $h=1,\ldots ,m$, one more sequence denoted by $(Z_{j}^{h}(n))_{n\in \ZZ }$ satisfying
\begin{equation}\label{relaR}
\varepsilon_{n+1}^{h}a_{n+1}Z_{j}^{h}(n+1)-b_nZ_{j}^{h}(n)+\frac{c_n}{\varepsilon_{n}^{h}}Z_{j}^{h}(n-1)=\lambda _h(j)Z_{j}^{h}(n),\quad n\in \ZZ,
\end{equation}
where, we assume that if for some $n_0$ and $h_0$, $\varepsilon_{n_0}^{h_0}=0$, then also $c_{n_0}=0$ and there is a number $d_{n_0}^{h_0}$  such that the identity (\ref{relaR}) still holds when we replace $c_{n_0}/\varepsilon_{n_0}^{h_0}$ by $d_{n_0}^{h_0}$.

Assume also that the polynomials $(p_n)_n$ are orthogonal with respect to a measure  and satisfy the three term recurrence relation ($p_{-1}=0$)
\begin{equation}\label{ttrr}
\lambda p_n(\lambda)=a_{n+1}p_{n+1}(\lambda)+b_np_n(\lambda)+c_np_{n-1}(\lambda), \quad n\ge 0.
\end{equation}
The measure $\mu$ might be degenerate, in which case for some $n_0$ we might have $a_{n_0}c_{n_0}=0$.
Define the auxiliary numbers $\xi _{n,i}^h$, $i\ge 0$, $n\in \ZZ $ and $h=1,\ldots ,m$, by
\begin{equation*}\label{defxi2}
\xi_{n,i}^h=\prod_{j=n-i+1}^{n}\varepsilon_{j}^{h},\quad i\ge 1,\quad \quad \xi_{n,0}^h=1,\quad\quad \xi_{n,i}^h=\frac{1}{\xi_{n-i,-i}^h},\quad i\leq-1.
\end{equation*}
For $i<0$, we take $\xi_{n,i}^h=\infty$ if $\xi_{n-i,-i}^h=0$ (what can happen if for some $n_0,h_0$, $\varepsilon_{n_0}^{h_0}=0$).
However, by assuming $\varepsilon_n^h\not =0$, $n\le 0$, $h=1,\cdots , m$, one can straightforwardly check that $\xi_{n,n+1}^h$ is finite for $n\le 0$.
Notice that for $x=n$, the number $\xi_{n,i}^h$ coincides with the number defined by (\ref{defxi}).

%%Given a finite set $G$ of $m$ positive integers,
%%$G=\{ g_1,\ldots , g_m\}$, assume that the $m$ numbers

Given a $m$-tuple $G$ of $m$ positive integers, $G=(g_1,\ldots , g_m)$, assume that the $m$ numbers
\begin{equation}\label{diffn}
\tilde g_h=\lambda _h(g_h),\quad h=1,\ldots ,m,
\end{equation}
are different. Call $\tilde G=\{ \tilde g_1,\ldots , \tilde g_m\}$. Consider finally the $m\times m$ Casorati determinant $\Omega _G$ defined by
\begin{equation*}\label{casort}
\Omega_G (n)=\det \left(\xi_{n,m-j}^lZ_{g_l}^l(n-j)\right)_{l,j=1}^m.
\end{equation*}

We then define the sequence of polynomials $(q_n^G)_n$ by
\begin{equation}\label{quso}
q_n^G(\lambda)=\begin{vmatrix}
p_n(\lambda) & -p_{n-1}(\lambda) & \cdots & (-1)^mp_{n-m}(\lambda) \\
\displaystyle\xi_{n,m}^1Z_{g_1}^1(n) & \xi_{n-1,m-1}^1Z_{g_1}^1(n-1) & \cdots &
Z_{g_1}^1(n-m) \\
               \vdots & \vdots & \ddots & \vdots \\
              \xi_{n,m}^mZ_{g_m}^m(n) & \displaystyle
               \xi_{n-1,m-1}^mZ_{g_m}^m(n-1) & \cdots &Z_{g_m}^m(n-m)
             \end{vmatrix}.
\end{equation}
Notice that if for each $h=1,\ldots , m$,  $Y_h(n)=Z_{g_h}^h(n)$,  is a polynomial in $n$ and for some $M\in \NN$, $\Omega_G (n)\not =0$ for $0\le n\le M$, then the polynomials $q_n^G$ (\ref{quso})
fit into the definition of the polynomials (\ref{qus})  in Theorem \ref{Teor1}, and hence they are eigenfunctions of an operator in the algebra $\A$.

The key to prove that the polynomials $(q_n^G)_n$ are orthogonal with respect to a measure $\tilde \rho $ are the following formulas.
Assume that $\varepsilon_n^h\not =0$, $n\le 0$, $h=1,\cdots , m$ and that there exists a constant $c_G\not =0$ such that
\begin{align}\label{foeq1}
\langle \tilde \rho,p_n\rangle &=(-1)^n c_G\sum_{i=1}^{m}\frac{\xi_{n,n+1}^i Z_{g_i}^i(n)}{\p _{\tilde G}'(\tilde g_i)Z_{g_i}^{i}(-1)},\quad n\ge 0,\\
\label{foeq2}
0&=\sum_{i=1}^{m}\frac{Z_{g_i}^i(n)}{\p _{\tilde G}'(\tilde g_i)\xi_{-1,-n-1}^iZ_{g_i}^{i}(-1)},\quad 1-m\le n<0,\\
\label{foeq3}
0&\not =\sum_{i=1}^{m}\frac{Z_{g_i}^i(-m)}{\p _{\tilde G}'(\tilde g_i)\xi_{-1,m-1}^iZ_{g_i}^{i}(-1)},
\end{align}
where $\tilde g_h$ are the $m$ different numbers (\ref{diffn}) and $\p_{\tilde G}(x)=\prod_{j=1}^m(x-\tilde g_i)$.

We then have the following version of Lemma 4.2 of \cite{ddI0}.

\begin{lemma}[] \label{lort2} Assume that $\varepsilon_n^h\not =0$, $n\le 0$, $h=1,\cdots , m$, and that there exist $M,N$ (each of them can be either a positive integer or infinity) such that $a_{n}c_n\not =0$ for $1\le n\le N$ and
$\Omega_G (n)\not =0$ for $0\le n\le M$. Assume also that (\ref{foeq1}), (\ref{foeq2}) and (\ref{foeq3}) hold.
Then the polynomials $q_n^G$, $0\le n\le \min\{M-1,N+m\}$, are orthogonal with respect to $\tilde \rho$ and have non-null norm.
\end{lemma}

The proof is completely analogous to that of Lemma 4.2 of \cite{ddI0} and it is omitted.

\subsection{Finite sets of positive integers}
We still need a  last ingredient for identifying the measure $\tilde \rho$ with respect to which the polynomials $(q_n^G)_n$ (\ref{quso}) are orthogonal.
The measures $\rho^{\F}_{\alpha,\beta,N}$ (\ref{mii}) in the Introduction  depends on certain finite sets $F_1, F_2$ and $F_3$ while the polynomials $(q_n^G)_n$ depend on the finite set $G$ (the degrees of the polynomials $Z$'s). The relationship between the sets $F$'s and $G$ will be given by the following transforms of finite sets of positive integers.

Consider the sets $\Upsilon$ and $\Upsilon _0$ formed by all finite sets of positive or nonnegative integers, respectively:
\begin{align*}
\Upsilon&=\{F:\mbox{$F$ is a finite set of positive integers}\} ,\\
\Upsilon_0&=\{F:\mbox{$F$ is a finite set of nonnegative integers}\} .
\end{align*}
We consider an involution $I$ in $\Upsilon$, and a family $J_h$, $h\ge 1$, of transforms from  $\Upsilon$ into  $\Upsilon _0$. For $F\in \Upsilon$ write $F=\{f_1,\ldots ,f_k\}$ with $f_i<f_{i+1}$, so that $f_k=\max F$. Then $I(F)$ and $J_h(F)$, $h\ge 1$, are defined by
\begin{align}\label{dinv}
I(F)&=\{1,2,\ldots, f_k\}\setminus \{f_k-f,f\in F\},\\ \label{dinv2}
J_h(F)&=\{0,1,2,\ldots, f_k+h-1\}\setminus \{f-1,f\in F\}.
\end{align}
Notice that $I$  is an involution: $I^2=Id$.

For the involution $I$, the bigger the holes in $F$ (with respect to the set $\{1,2,\ldots , f_k\}$), the bigger the involuted set $I(F)$.
Here it is a couple of examples
$$
I(\{ 1,2,3,\ldots ,k\})=\{ k\},\quad \quad I(\{1, k\})=\{ 1,2,\ldots, k-2, k\}.
$$
Something similar happens for the transform $J_h$ with respect to $\{0,1,\ldots , f_k+h-1\}$.

Notice that
$$
\max F=\max I(F), \quad h-1+\max F=\max J_h(F),
$$
and if $n_F$ denotes the cardinal of $F$, we also have
\begin{equation}\label{neci}
n_{I(F)}=f_k-n_F+1,\quad n_{J_h(F)}=f_k+h-n_F.
\end{equation}

For a trio $\F=(F_1,F_2,F_3)$ of finite sets of positive integers, we will write $F_i=\{ f_1^{ i\rceil },\cdots ,
f_{k_i}^{i\rceil}\}$, $i=1,2,3$, with
$f_j^{i\rceil}<f_{j+1}^{i\rceil}$ (the use of, for instance, $f_j^2$ to describe elements of $F_2$ is confusing because it looks like a square, this is the reason why we use the notation $f_j^{2\rceil}$).

\section{Dual-Hahn polynomials and their $\D$-operators}\label{sch}
We start with some basic definitions and facts about dual Hahn and Hahn polynomials, which we will need later.

For $\alpha $ and $\beta$ real numbers, we write $\lambda^{\alpha,\beta}(x)=x(x+\alpha+\beta+1)$. To simplify the notation we sometimes write $\lambda(x)=\lambda^{\alpha,\beta}(x)$.

For $\alpha \not =-1,-2,\cdots $ we write $(R_{n}^{\alpha,\beta,N})_n$ for the sequence of dual Hahn polynomials defined by
\begin{equation}\label{Mxpol}
R_{n}^{\alpha,\beta,N}(\lambda)=\frac{1}{n!}\sum _{j=0}^n(-1)^j\frac{(-n)_j(-N+j)_{n-j}}{(\alpha+1)_jj!}\prod_{i=0}^{j-1}(\lambda-i(\alpha+\beta+1+i))
\end{equation}
(we have taken a slightly different normalization from the one used in \cite{KLS}, pp, 234-7 from where
the next formulas can be easily derived). Notice that $R_{n}^{\alpha,\beta,N}$ is always a polynomial in $\lambda$ of degree $n$
and leading coefficient equal to $1/((\alpha+1)_nn!)$.
Using that $\prod_{i=0}^{j-1}(\lambda^{\alpha,\beta}(x)-i(\alpha+\beta+1+i))=(-x)_j(x+\alpha+\beta+1)_j$, we get the hypergeometric representation
$$
R_{n}^{\alpha,\beta,N}(\lambda^{\alpha,\beta}(x))=\frac{(-N)_n}{n!}\pFq{3}{2}{-n,-x,x+\alpha+\beta+1}{\alpha+1,-N}{1}.
$$
When $N$ is a positive integer then the polynomial $R_{n}^{\alpha,\beta,N}(\lambda)$ for $n\ge N+1$ is always divisible by $\prod_{i=0}^{N}(\lambda+i(-\alpha-\beta-1-i))$. Hence
\begin{equation}\label{mdc1}
R_{n}^{\alpha,\beta,N}(\lambda^{\alpha,\beta}(i))=0,\quad n\ge N+1, i=0,\cdots, N.
\end{equation}
Dual Hahn polynomials are eigenfunctions of the  second order difference operator
\begin{equation}\label{defbc}
\Gamma=B(x)\Sh_{x,1}-(B(x)+D(x))\Sh_{x,0}+D(x)\Sh_{x,-1},\quad \Gamma(R_n)=nR_n,
\end{equation}
where
\begin{align*}
B(x)&=-\frac{(x+\alpha+1)(x+\alpha+\beta+1)(N-x)}{(2x+\alpha+\beta+1)(2x+\alpha+\beta+2)},\\
D(x)&=-\frac{x(x+\alpha+\beta+N+1)(x+\beta)}{(2x+\alpha+\beta)(2x+\alpha+\beta+1)}
\end{align*}
(to simplify the notation we remove the parameters in some formulas). As in Section 3, the shift operators in $\PP^\lambda$ act in $x$: $\Sh_{x,j}(p)=p(\lambda(x+j))$.
In particular, this implies that $\Gamma\in \A^\lambda$, where $\A^\lambda$ is the algebra defined by (\ref{defal}).

Dual Hahn polynomials satisfy the three term recurrence formula ($R_{-1}=0$)
\begin{equation}\label{Mxttrr}
\lambda R_n(\lambda)=a_{n+1}R_{n+1}(\lambda)+b_nR_n(\lambda)+c_nR_{n-1}(\lambda), \quad n\ge 0
\end{equation}
where
\begin{align*}
a_n&=n(n+\alpha),\\
b_n&=-(n+\alpha+1)(n-N)-n(n-\beta-N-1),\\
c_n&=(n-\beta-N-1)(n-N-1).
\end{align*}
Hence, when $N$ is not a positive integer and $\alpha ,-\beta-N-1 \not =-1,-2,\cdots $, they are always orthogonal with respect to a moment functional $\rho_{\alpha,\beta,N}$. When $N$ is a positive integer and $\alpha ,\beta \not =-1,-2,\cdots -N $, $\alpha+\beta \not=-1,\cdots, -2N-1$, we have
\begin{equation*}\label{MXw}
\rho_{\alpha,\beta,N}=\sum _{x=0}^N w_{*;\alpha,\beta,N}(x)\delta_{x},
\end{equation*}
where
\begin{equation}\label{masdh}
w_{*;\alpha,\beta,N}(x)=\frac{(2x+\alpha+\beta+1)(\alpha+1)_x(-N)_xN!}{(-1)^x(x+\alpha+\beta+1)_{N+1}(\beta+1)_xx!},
\end{equation}
and
\begin{equation}\label{norme}
\langle R_n^{\alpha,\beta,N}(\lambda),R_n^{\alpha,\beta,N}(\lambda)\rangle =\frac{(-N)_n^2}{n!^2\binom{\alpha+n}{n}\binom{\beta+N-n}{N-n}},\quad  n\in \NN .
\end{equation}
Notice that $\langle R_n^{\alpha,\beta,N},R_n^{\alpha,\beta,N}\rangle\not =0$ only for $0\le n\le N$.
The moment functional $\rho_{\alpha,\beta,N}$ can be represented by either a positive or a negative measure only when $N$ is a positive integer
and either $-1<\alpha,\beta$ or $\alpha,\beta<-N$, respectively.

Dual Hahn polynomials satisfy the following identities
\begin{align}\label{dhmp1}
\sum_{j=0}^nR_{j}^{\alpha,\beta,N}(\lambda)&=R_{n}^{\alpha,\beta,N-1}(\lambda),\\
\label{dhmp2}
\Delta_x R_{n}^{\alpha,\beta,N}(\lambda(x))&=\frac{2x+\alpha+\beta+2}{\alpha+1}R_{n-1}^{\alpha+1,\beta,N-1}(\lambda^{\alpha+1,\beta}(x)),\\
\label{dhmp3}
(1-t)^{N-x}\pFq{2}{1}{-x,-x-\beta}{\alpha+1}{t}&=\sum_{n=0}^\infty R_{n}(\lambda(x))t^n,\quad  N\not \in \NN, x\in \NN, \vert t\vert <1.
\end{align}
\bigskip
For $\alpha+\beta \not =-1,-2,\cdots $ we write $(h_{n}^{\alpha,\beta,N})_n$ for the sequence of Hahn polynomials defined by
\begin{equation}\label{hxpol}
h_{n}^{\alpha,\beta,N}(x)=\sum _{j=0}^n\frac{(-n)_j(n+\alpha+\beta+1)_j(-N+j)_{n-j}(\alpha+j+1)_{n-j}(-x)_j}{j!}
\end{equation}
(we have taken a slightly different normalization from the one used in \cite{KLS}, pp, 234-7). Notice that $h_{n}^{\alpha,\beta,N}$ is always a polynomial of degree $n$.
A straightforward computation shows the hypergeometric representation
$$
h_{n}^{\alpha,\beta,N}(x)=(-N)_n(\alpha+1)_n\pFq{3}{2}{-n,-x,n+\alpha+\beta+1}{\alpha+1,-N}{1}.
$$
Hahn polynomials satisfy the following second order difference equation
\begin{equation}\label{hdeq}
B(x)h_n(x+1)-(B(x)+D(x))h_n(x)+D(x)h_n(x-1)=\lambda(n)h_n,
\end{equation}
where
\begin{align*}
B(x)&=(x+\alpha+1)(x-N),\\
D(x)&=x(x-\beta-N-1).
\end{align*}
\bigskip

In the following Lemma (which will be proved in Section \ref{sproofs}), we include the three $\D$-operators we have found for dual Hahn polynomials.

\begin{lemma}\label{dhdo}
The sequences given by
\begin{align}
\label{eps1}\varepsilon_{n,1}&=-1,\\
\label{eps2}\varepsilon_{n,2}&=\frac{\beta+N-n+1}{\alpha +n},\\
\label{eps3}\varepsilon_{n,3}&=\frac{N-n+1}{\alpha +n},
\end{align}
define three $\mathcal{D}$-operators (see \eqref{Dh}) for the dual Hahn polynomials and the algebra $\A^\lambda$ of operators defined by
\eqref{defal}. More precisely
\begin{align*}
\mathcal{D}_1&=-\frac{(x+\alpha+1)(x+\alpha+\beta+1)}{(2x+\alpha+\beta+1)(2x+\alpha+\beta+2)}\Delta_x+
\frac{x(x+\beta)}{(2x+\alpha+\beta)(2x+\alpha+\beta+1)}\nabla_x,\\
\mathcal{D}_2&=\frac{(x+\alpha+1)(N-x)}{(2x+\alpha+\beta+1)(2x+\alpha+\beta+2)}\Delta_x+
\frac{(x+\beta)(x+\alpha+\beta+N+1)}{(2x+\alpha+\beta)(2x+\alpha+\beta+1)}\nabla_x,\\
\mathcal{D}_3&=\frac{(x+\alpha+\beta+1)(N-x)}{(2x+\alpha+\beta+1)(2x+\alpha+\beta+2)}\Delta_x+
\frac{x(x+\alpha+\beta+N+1)}{(2x+\alpha+\beta)(2x+\alpha+\beta+1)}\nabla_x.
\end{align*}
\end{lemma}
We can apply Theorem \ref{Teor1} to produce from arbitrary polynomials $Y_j$, $j\ge 0$, a large class of sequences of polynomials $(q_n)_n$ satisfying higher order difference equations. But only for a convenient choice of the polynomials $Y_j$, $j\ge 0$, these polynomials $(q_n)_n$ are also orthogonal with respect to a measure. As we wrote in the Introduction, when the sequence $(p_n)_n$ is the dual Hahn polynomials a very nice symmetry between the family $(p_n)_n$ and the polynomials $Y_j$'s appears. Indeed, the polynomials $Y_j$ can be chosen as Hahn polynomials with parameters depending on the $\D$-operator $\D_h$. This symmetry is given by the recurrence relation (\ref{relaR}), where $Y_i=Z_i^h$.

\begin{lemma}\label{lemRme} Consider the Hahn polynomials
\begin{align*}
Z_{j}^{1}(x)&=h_j^{\beta+N+1,\alpha+N+1,-2-N}(-x-1),\quad j\ge 0,\\
Z_{j}^{2}(x)&=h_j^{-\alpha,-\beta,-2-N}(-x-1),\quad j\ge 0,\\
Z_{j}^{3}(x)&=h_j^{-\alpha,\beta,-\beta-2-N}(-x-1),\quad j\ge 0.
\end{align*}
Then they satisfy the recurrence (\ref{relaR}), where $(a_n)_{n\in \ZZ }$, $(b_n)_{n\in \ZZ }$, $(c_n)_{n\in \ZZ }$ are the sequences of coefficients in the three term recurrence relation for the dual Hahn polynomials $(R_n^{\alpha,\beta,N})_n$ (\ref{Mxttrr}) and
\begin{align*}
&\varepsilon_{n}^{1}=-1, \quad &\lambda_1(j)&=-\lambda^{\alpha,\beta}(j+N+1),\\
&\varepsilon_{n}^{2}=\frac{\beta+N-n+1}{\alpha +n}, \quad &\lambda_2(j)&=-\lambda^{\alpha,\beta}(-j-1),\\
&\varepsilon_{n}^{3}=\frac{N-n+1}{\alpha +n}, \quad &\lambda_3(j)&=-\lambda^{\alpha,\beta}(j-\alpha),\\
\end{align*}
respectively.
\end{lemma}

\begin{proof}
We only prove the first case. The recurrence relation (\ref{relaR}) is then
\begin{align*}
-(n+1)&(n+1+\alpha)Z_{j}^{1}(n+1)+((n+\alpha+1)(n-N)+n(n-\beta-N-1))Z_{j}^{1}(n)\\
&-(n-\beta-N-1)(n-N-1)Z_{j}^{1}(n-1)=-\lambda^{\alpha,\beta}(j+N+1)Z_{j}^{1}(n)
\end{align*}
where $n\in \ZZ$.
But this follows straightforwardly by writing $x=-n-1$ in the second order difference equation (\ref{hdeq}) for the Hahn polynomials $(h_j^{\beta+N+1,\alpha+N+1,-2-N})_j$.
\end{proof}

\section{Bispectral dual Hahn polynomials}
In this section we put together all the ingredients showed in the previous Sections to construct bispectral dual Hahn polynomials.
Along this section we assume that $N$ is a positive integer. This condition in necessary for the existence of a positive weight for the dual Hahn polynomials, and only in this case we have an explicit expression of that weight. However, this condition is not needed in our construction and hence the results in this Section are also valid when $N$ is not a positive integer (once one has adapted the constrains on the parameters $\alpha$, $\beta$ and $N$).

Before stating the main result, we need some notation.
Given a trio $\U=(U_1,U_2,U_3)$ of finite sets of nonnegative integers, we will write $m_j$
for the number of elements of $U_j$, $j=1,2,3$, $m=m_1+m_2+m_3$ and
\begin{equation}\label{ppmm2}
\UU_1=\{1,\cdots, m_1\},\quad \UU_2=\{m_1+1,\cdots, m_1+m_2\},\quad \UU_3=\{m_1+m_2+1,\cdots, m\}.
\end{equation}
We write $U_j=\{u_i^{j\rceil},i\in \UU_j\}$.

Since we  have three $\D$-operators for dual Hahn polynomials, we make a partition of the  indices in Theorem \ref{Teor1}
and take
\begin{equation}\label{defvardh}
\varepsilon _n^h=\begin{cases} -1,& \mbox{for $h\in \UU_1$,}\\
\frac{\beta+N-n+1}{\alpha +n},& \mbox{for $h\in \UU_2$,}\\
\frac{N-n+1}{\alpha +n},& \mbox{for $h\in \UU_3$.}
 \end{cases}
\end{equation}
In particular, the auxiliary sequences of numbers $\xi_{x,i}^h$, $h=1,\cdots, m$, $i\in \ZZ$, (see (\ref{defxi})) are then the following
 rational functions of $x$
\begin{equation}\label{defxii}
\xi_{x,i}^h=\begin{cases}(-1)^i,& \mbox{for $h\in \UU_1$,}\\
\frac{(\beta+N-x+1)_i}{(\alpha+x-i+1)_i},& \mbox{for $h\in \UU_2$,}\\
\frac{(N-x+1)_i}{(\alpha+x-i+1)_i},&  \mbox{for $h\in \UU_3$.}
\end{cases}
\end{equation}
For $i\in \ZZ$, we finally write
$$
\ZZ_{i}=\{ j\in \ZZ:j\le i\}.
$$

We are now ready to establish  the main Theorem of this paper.

\begin{theorem}\label{t6.2} Let $\F=(F_1,F_2,F_3)$ be a trio of finite sets of positive integers (the empty set is allowed, in which case we take $\max F=-1$). For $h_1,h_3\ge 1$, consider the trio $\U=(U_1,U_2,U_3)$ whose elements are the transformed sets $J_{h_j}(F_j)=U_j=\{ u_{i}^{j\rceil}: i\in \UU_j\}$, $j=1,3$, and $I(F_2)=U_2=\{ u_{i}^{2\rceil}: i\in \UU_2\}$, where the involution $I$ and the transform $J_h$ are defined by (\ref{dinv}) and (\ref{dinv2}), respectively. Define $m=m_1+m_2+m_3$.
Let $\alpha$ and $\beta$ be real numbers satisfying
\begin{equation}\label{appm1}
\alpha\not \in \ZZ_{f_{2,M}+f_{3,M}+h_3}, \quad \beta\not \in \ZZ_{f_{2,M}},\quad \alpha+\beta\not\in \ZZ_{2f_{2,M}+1},
\end{equation}
where we denote by $f_{i,M}$ the maximum element in $F_i$, $i=1,2,3$.
In addition, we assume that
\begin{equation}\label{appm2}
\alpha+\beta-1\not \in \NN, \quad \mbox{if $F_2\not =\emptyset$,\quad and \quad  $\alpha-\beta-1\not \in \NN,$ \quad if $F_3\not =\emptyset$}.
\end{equation}
Consider the dual Hahn and Hahn polynomials $(R_n^{\alpha,\beta,N})_n$ (\ref{Mxpol}) and $(h_n^{\alpha,\beta,N})_n$ (\ref{hxpol}), respectively. Assume that $\Omega_{\alpha,\beta,N}^{\U} (n)\not =0$ for $0\le n\le N+m_1+m_2+1$ where the $m\times m$ Casorati determinant $\Omega _{\alpha,\beta,N}^{\U}$ is defined by
\begin{align}\label{defom}
\Omega _{\alpha, \beta, N}^\U(x)&=  \left|
  \begin{array}{@{}c@{}lccc@{}c@{}}
    & &&\hspace{-1.3cm}{}_{j=1,\ldots , m} \\
    \dosfilas{ (-1)^{j}h_{u}^{\beta +N+1,\alpha +N+1,-2-N}(-x+j-1) }{u\in U_1} \\
    \dosfilas{ \frac{(\beta +N-x+j+1)_{m-j}}{(\alpha+x-m+1)_{m-j}}h_{u}^{-\alpha,-\beta,-2-N}(-x+j-1) }{u\in U_2}
    \\
    \dosfilas{ \frac{(N-x+j+1)_{m-j}}{(\alpha+x-m+1)_{m-j}}h_{u}^{-\alpha,\beta,-\beta-2-N}(-x+j-1) }{u\in U_3}
  \end{array}\right|.
\end{align}
We then define the sequence of polynomials $q_n$, $n\ge 0$, by
\begin{equation}\label{qusch}
q_n(\lambda)=\left|
  \begin{array}{@{}c@{}lccc@{}c@{}}
   &(-1)^{j-1}R_{n+1-j}^{\alpha,\beta,N}(\lambda) &&\hspace{-1.3cm}{}_{j=1,\ldots , m+1} \\
    \dosfilas{ (-1)^{j-1}h_{u}^{\beta +N+1,\alpha +N+1,-2-N}(-n+j-2) }{u\in U_1} \\
    \dosfilas{ \frac{(\beta+N-n+j)_{m+1-j}}{(\alpha+n-m+1)_{m+1-j}}h_{u}^{-\alpha,-\beta,-2-N}(-n+j-2) }{u\in U_2}
    \\
    \dosfilas{ \frac{(N-n+j)_{m+1-j}}{(\alpha+n-m+1)_{m+1-j}}h_{u}^{-\alpha,\beta,-\beta-2-N}(-n+j-2) }{u\in U_3}
  \end{array}\right|
\end{equation}
Then

\noindent
(1) The polynomials $q_n$, $0\le n\le N+m_1+m_2$, are orthogonal and have non-null norm with respect to the measure
\begin{align*}
\tilde \rho^{\F,h_1,h_3}_{\alpha,\beta,N}&=\prod_{f\in F_1}(\lambda(x)-\lambda(N+f))\prod_{f\in F_2}(\lambda(x)-\lambda(-f_{2,M}-1+f))\\&\quad\quad \quad \times
\prod_{f\in F_3}(\lambda(x)-\lambda(f-\alpha-1)) \rho_{\tilde \alpha,\tilde \beta ,\tilde N}(x+f_{2,M}+1),
\end{align*}
where
\begin{equation}\label{abnt}
\tilde \alpha=\alpha -f_{2,M}-f_{3,M}-h_3-1,\quad \tilde \beta=\beta -f_{2,M}+f_{3,M}+h_3-1,
\quad \tilde N=N+f_{1,M}+f_{2,M}+h_1+1.
\end{equation}
(2) The polynomials $q_n$, $0\le n\le N+m_1+m_2$, are eigenfunctions of a higher order difference operator of the form (\ref{doho}) with
$$
-s=r=\sum_{f\in F_2}f-\sum_{f\in F_1,F_3}f-\sum_{i=1}^3\binom{k_i}{2}+k_1(f_{1,M}+h_1)+k_3(f_{3,M}+h_3)+1
$$
(which can be explicitly constructed using Theorem \ref{Teor1}).
\end{theorem}

\begin{proof}
First of all, notice that we have performed a straightforward normalization of the polynomials $q_n$, $0\le n\le N+m_1+m_2$ (with respect to (\ref{qus})).

Notice that the assumption (\ref{appm1}) on the parameters $\alpha$ and $\beta$ implies that
$$
\tilde \alpha, \tilde \beta \not =-1,\cdots , -\tilde N, \tilde \alpha + \tilde \beta \not =-1,\cdots , -2\tilde N -1,
$$
and hence the dual Hahn weight $\rho_{\tilde \alpha,\tilde \beta ,\tilde N}(x+f_{2,M}+1)$ is well defined and its support is
$\{-f_{2,M}-1,\cdots ,N+f_{1,M}+h_1\}$. Using the assumptions (\ref{appm1}) on the parameters $\alpha$ and $\beta$, we deduce that the support of the measure $\tilde \rho^{\F,h_1,h_3}_{\alpha,\beta,N}$ is
$$
\{-f_{2,M}-1,\cdots ,N+f_{1,M}+h_1\}\setminus \Big((N+F_1)\cup (-f_{2,M}-1+F_2)\Big) .
$$
Notice that the support is formed by $N+f_{1,M}+f_{2,M}+h_1+2-k_1-k_2$ integers. Taking into account  that $U_1=J_{h_1}(F_1)$, $U_2=I(F_2)$ and (\ref{neci}) we get
$$
N+f_{1,M}+f_{2,M}+h_1+2-k_1-k_2=N+m_1+m_2+1.
$$
Before going on with the proof we comment on the assumption that $\Omega_{\alpha,\beta,N}^{\U} (n)\not =0$ for $0\le n\le N+m_1+m_2+1$. If $F_3\not =0$, since the sequence $\varepsilon _n^h$, $h\in \UU_3$, vanish for $n=N+1$, it is no difficult to see that $\Omega_{\alpha,\beta,N}^{\U} (n) =0$ for $N+m_1+m_2+2\le n\le N+m$ (the proof is similar to that of the first part of Lemma \ref{lgp1}). If $F_3=\emptyset$, the situation is different, and, except for exceptional values of the parameters $\alpha, \beta$ and $N$, we have $\Omega_{\alpha,\beta,N}^{\U} (n)\not =0$ for all $n\ge 0$.
In this cases, the polynomials $(q_n)_n$ are defined for all $n\ge 0$ and always have degree $n$ (in $\lambda$). However, it is not difficult to see that for $n\ge N+m_1+m_2+1$, the polynomial $q_n(\lambda(x))$ vanishes in the support of $\tilde \rho^{\F,h_1,h_3}_{\alpha,\beta,N}$. Hence it is still orthogonal with respect to this measures but has null norm.
This is completely analogous to the situation with the dual Hahn polynomials $R_n(\lambda)$, which are defined for all $n\ge 0$ (except when $\alpha =-1,-2,\cdots $) and always have degree $n$. But if $n\ge N+1$, they vanish in the support of its weight (see (\ref{mdc1})).

To prove (1) of the Theorem, we use the strategy of the Section \ref{ssi}.

We need some notation. Write $Z^h_j$, $h=1,\ldots ,m, j\ge 0$, for the polynomials
\begin{equation}\label{defz}
Z^h_j(x)=\begin{cases} h_j^{\beta+N+1,\alpha+N+1,-2-N}(-x-1),&h\in \UU_1,\\
h_j^{-\alpha,-\beta,-2-N}(-x-1),&h\in \UU_2,\\
h_j^{-\alpha,\beta,-\beta-2-N}(-x-1),&h\in \UU_3.\end{cases}
\end{equation}
The assumptions (\ref{appm1}) and (\ref{appm2}) on the parameters $\alpha$ and $\beta$ implies that these Hahn polynomials are well defined and have degree $j$.
%%$G$ for the set $G=H\cup K=\{g_1,\ldots, g_m\}$
Denote by $G$ and $\tilde G$ the $m$-tuples
\begin{align}\label{defgg}
G&=(u_{1}^{1\rceil},\ldots, u_{m_1}^{1\rceil},u_{1}^{2\rceil},\ldots, u_{m_2}^{2\rceil},u_{1}^{3\rceil},\ldots, u_{m_3}^{3\rceil})=(g_1,\ldots, g_m),\\\label{defggt}
\tilde G&=(\tilde g_1,\ldots, \tilde g_m),
\end{align}
where
$$
\tilde g_i=\begin{cases} -\lambda(u_{i}^{1\rceil}+N+1),& i\in \UU_1, \\-\lambda(-u_{i}^{2\rceil}-1),& i\in \UU_2,
\\-\lambda(u_{i}^{3\rceil}-\alpha),& i\in \UU_3. \end{cases}
$$
Finally, write $\p _{\tilde G}$ for the polynomial
$$
\p_{\tilde G}(x)=\prod_{u\in U_1}(x+\lambda(u+N+1))\prod_{u\in U_2}(x+\lambda(-u-1))\prod_{u\in U_3}(x+\lambda(u-\alpha)).
$$
It is easy to see that $\p_{\tilde G}$ has simple roots if and only if
\begin{align}\label{yas}
&\alpha+\beta +N+1+u-v,\alpha+N+1+u-w,\beta-v+w\not =0,\\\nonumber
&N+2+u+v,N+\beta +2+u+w,\alpha-1-v-w\not =0,
\end{align}
for $u\in U_1,v\in U_2,w\in U_3$. These constrains follow easily from the assumptions (\ref{appm1})  on the parameters $\alpha$ and $\beta$. Hence, $\p_{\tilde G}$ has simple roots.

Proceeding as in the proof of Theorem 1.1 in \cite{ddI0}, one can prove that
\begin{align}\label{foeq1ch}
c^{\F,h_1,h_3}_{\alpha,\beta,N}\langle \tilde \rho^{\F,h_1,h_3}_{\alpha,\beta,N},R_n^{\alpha,\beta,N}\rangle &=(-1)^n\sum_{i=1}^{m}\frac{\xi^i_{n,n+1}Z^i_{g_i}(n)}{\p _{\tilde G}'(\tilde g_i)Z^i_{g_i}(-1)},\quad n\ge 0,\\\label{foeq2ch}
\sum_{i=1}^{m}\frac{Z^i_{g_i}(n)}{\xi^i_{-1,-n-1}\p _{\tilde G}'(\tilde g_i)Z^i_{g_i}(-1)}&=0,\quad 1-m\le n<0,\\\label{foeq3ch}
\sum_{i=1}^{m}\frac{Z^i_{g_i}(-m)}{\xi^i_{-1,m-1}\p _{\tilde G}'(\tilde g_i)Z^i_{g_i}(-1)}&\not =0,
\end{align}
where $c^{\F,h_1,h_3}_{\alpha,\beta,N}$ is the constant independent of $n$ given by
$$
c^{\F,h_1,h_3}_{\alpha,\beta,N}=\frac{(-1)^{k_1+m_2+m_3-1}
(\beta-f_{2,M}+f_{3,M}+h_3)_{N+f_{2,M}-f_{3,M}-h_3+2}(N+1)!}
{(\alpha-f_{2,M}-f_{3,M}-h_3)_{f_{2,M}+f_{3,M}+h_3+1}
((N+f_{1,M}+f_{2,M}+h_1+1)!)^2},
$$
and  $\xi_{x,y}^i$ are defined by (\ref{defxii}).

From the recurrence relation for the dual Hahn polynomials, we get that $a_{n}c_n\not =0$ for $0\le n\le N$, but $c_{N+1}=0$. It is also easy to check that $\varepsilon_n^h\not =0$, $h=1,\cdots, m$, when $n$ is a negative integer.
Since we assume that $\Omega_{\alpha,\beta,N}^{\U} (n)\not =0$ for $0\le n\le N+m_1+m_2+1$, the orthogonality of the polynomials $q_n$, $0\le n\le N+m_1+m_2$, with respect to $\tilde \rho_{\alpha,\beta,N}^{\U}$ is now a consequence of the Lemmas \ref{lemRme}, \ref{lort2} and the identities (\ref{defvardh}) and (\ref{defxii}). They have also non-null norm.

We now prove (2) of the Theorem.
Using (\ref{defxii}), it is straightforward to see that $\Omega_{\alpha,\beta,N}^\U(x)$ coincides with the (quasi) Casorati determinant
$$
(-1)^{mm_1}\det \left(\xi_{x-j,m-j}^lY_l(x-j)\right)_{l,j=1}^m,
$$
where $Y_l(x)=Z^l_{g_l}(x)$, and $Z^l_j$, $l=1,\cdots, m, j\ge 0$, and $G=\{g_1,\cdots, g_m\}$ are defined by (\ref{defz}) and (\ref{defgg}), respectively.
Consider the particular case of the polynomial $P$ (\ref{def2p}) in Lemma \ref{lgp1}
for $Y_l(x)=Z^l_{g_l}(x)$ (and denote it again by $P$), and write $S$ for the rational function
\begin{equation}\label{umy}
S(x)=\frac{(-1)^{\binom{m}{2}+m_1}(\alpha+x-m+1)_{m-1}^{m_2+m_3}}{p(x)},
\end{equation}
where $p$ is the polynomial (\ref{def1p}) in Lemma \ref{lgp1}.
A simple computation shows that $S(x)\Omega_{\alpha,\beta,N}^\U(x)=P(x)$.

Write now $M_h$, $\Psi_j^h$, $h,j=1,\cdots , m$, for the rational functions
\begin{align*}
M_h(x)&=\sum_{j=1}^m(-1)^{h+j}\xi_{x,m-j}^hS(x+j)
\det\left(\xi_{x+j-r,m-r}^{l}Y_l(x+j-r)\right)_{l\in \II_h;r\in \II_j},\\
\Psi_j^h(x)&=\xi_{x-j,m-j}^hS(x)
\det\left(\xi_{x-r,m-r}^{l}Y_l(x-r)\right)_{l\in \II_h;r\in \II_j},
\end{align*}
where $\II_h=\{1,2,\ldots,m\}\setminus\{h\}$.
Again a simple computation using (\ref{defxii}) shows that $\Psi_j^h$, $h,j=1,\cdots , m$, are polynomials in $x$. Hence Lemma \ref{lgp2} gives that
$M_h$ is also a polynomial in $x$.

Since the sequences $(\varepsilon _n^h)_n$ (\ref{defvardh}) generate the $\D $-operators in Lemma \ref{dhdo} for the dual Hahn polynomials, we get, as a direct consequence of Theorem \ref{Teor1}, that the polynomials $q_n$, $0\le n\le N+m_1+m_2$, are eigenfunctions of a higher order difference operator $D_{q,S}$ in the algebra $\A^\lambda$ (\ref{defal}), explicitly given by (\ref{Dq}).

We now compute the order of $D_{q,S}$. Since $S\Omega _{\alpha, \beta, N}^\U=P$, Lemma  \ref{lgp1} gives that the degree of $S\Omega _{\alpha, \beta, N}^\U$ is $d=\sum_{u\in U_1,U_2,U_3}u-\sum_{i=1}^3\binom{m_i}{2}$ (notice that the assumption (\ref{yas0}) in Lemma \ref{lgp1} are just (\ref{yas}) above).
Hence the polynomial $P_S$ defined by $P_S(x)-P_S(x-1)=S(x)\Omega _{\alpha, \beta, N}^\U(x)$ has degree $d+1$.
Taking into account that the $m$-tuple $G$ (\ref{defgg}) is formed by the sets $J_{h_1}(F_1)$, $I(F_2)$ and $J_{h_3}(F_3)$, the definitions of the involution $I$ (\ref{dinv}) and the transform $J_h$ (\ref{dinv2}) give
\begin{align*}
&\sum_{u\in U_1,U_2,U_3}u-\sum_{i=1}^3\binom{m_i}{2}+1\\ &\quad \qquad=
\sum_{f\in F_2}f-\sum_{f\in F_1,F_3}f-\sum_{i=1}^3\binom{k_i}{2}+k_1(f_{1,M}+h_1)+k_3(f_{3,M}+h_3)+1\\&
\quad \qquad =r.
\end{align*}
That is, $P_S$ is a polynomial of degree $r$. Consider the coefficients $B$ and $D$ of $\Sh_{x,1}$ and $\Sh_{x,-1}$ in the second order difference operator $\Gamma$ for the Dual Hahn polynomials (\ref{defbc}). We then deduce that
the operator $P_S(\Gamma)$ has the form
$$
\sum _{l=-r}^{r}\tilde h_{l}(x)\Sh_{x,l},
$$
where $\tilde h_{r}(x)=u_1\prod_{j=0}^{r-1}B(x+j)$, $\tilde h_{-r}(x)=u_1\prod_{j=0}^{r-1}D(x-j)$ and $u_1$ denotes the leading coefficient of the polynomial $P_S$.  Using (\ref{defbc}), we deduce that both $\tilde h_{-r}$ and $\tilde h_r$ are rational functions whose numerators are polynomials of degree $3r$ and whose denominators are polynomials of degree $2r$.

Consider now the coefficients $\tilde B_h$ and $\tilde D_h$ of $\Delta_{x,1}$ and $\nabla_{x,-1}$ in any of the $\D$-operators $\D_h$ for the Dual Hahn polynomials (see Lemma \ref{dhdo}). Using Lemmas \ref{lgp2} and \ref{lgp1}, we can conclude that the polynomials $M_h$ (\ref{emeiexp}) have degree at most
$v_h=r-g_h$. Since $Y_h$ has degree $g_h$, this shows that the operator $M_h(\Gamma)\D_h Y_h(\Gamma)$ has the form
$$
\sum _{l=-r}^{r}\hat h_{l}(x)\Sh_{x,l},
$$
where
\begin{align*}
\tilde h_{r}(x)&=u_2u_3\tilde B_h(x+v_h)\prod_{j=0;j\not =v_h}^{r-1}B(x+j),\\
\tilde h_{-r}(x)&=u_2u_3\tilde D_h(x-v_h)\prod_{j=0;j\not =v_h}^{r-1}D(x-j),
\end{align*}
$u_2$ is the leading coefficient of $Y_h$ and
$u_3$ is the coefficient of $x^{v_h}$ in $M_h$. As before, we deduce that both $\hat h_{-r}$ and $\hat h_r$ are rational functions whose numerators are polynomials of degree $3r-1$ and whose denominators are polynomials of degree $2r$.

To complete the proof of (2) of the Theorem it is enough to take into account the expression of $D_{q,S}$ given by (\ref{Dq}).
\end{proof}

Notice that we can generate more higher order difference operators with respect to which the polynomials $(q_n)_n$ are eigenfunctions by choosing a polynomial $p$, considering the rational function $S_p=pS$, where $S$ is defined by (\ref{umy}), and proceeding as in the proof of (2) of the previous Theorem. We guess that using this approach one can generate the whole algebra of difference operators having the polynomials $(q_n)_n$ as eigenfunctions (except for some exceptional values of the parameters $\alpha,\beta $ and $N$).

\begin{corollary}\label{jodme}
Let $\F=(F_1,F_2,F_3)$  be a trio of finite sets of positive integers (the empty set is allowed, in which case we take $\max F=-1$). Let $\alpha$ and $\beta$ be real numbers satisfying
\begin{align*}
&\alpha, \alpha+\beta\not \in \ZZ_{-1}, \quad \beta\not \in \ZZ_{f_{3,M}},
\\
\alpha+\beta+2f_{2,M}+1\not \in \NN, \quad &\mbox{if $F_2\not =\emptyset$,\quad and \quad  $\alpha-\beta+2f_{3,M}+1\not \in \NN,$ \quad if $F_3\not =\emptyset$}.
\end{align*}
Consider the weight $\rho _{\alpha,\beta,N}^{\F}$ defined by
\begin{equation}\label{udspm}
\rho _{\alpha,\beta,N}^{\F}=\prod_{f\in F_1}(\lambda-\lambda(N-f))\prod_{f\in F_2}(\lambda-\lambda(f))\prod_{f\in F_3}(\lambda-\lambda(f-\beta))\rho _{\alpha,\beta,N},
\end{equation}
where $\rho _{\alpha,\beta,N}$ is the dual Hahn weight. Assume that
$$
\Omega_{\alpha+f_{2,M}+f_{3,M}+2,\beta+f_{2,M}-f_{3,M},
N-f_{1,M}-f_{2,M}-2}^{\U}(n)\not =0, \quad 0\le n\le N+m_1+m_2+1,
$$
where $\U=(I( F_1),I(F_2),I(F_3))$.
Then the measure $\rho _{\alpha,\beta,N}^{\F}$ has associated a sequence of orthogonal polynomials and they are eigenfunctions of a higher order difference operator of the form (\ref{doho}) with
$$
-s=r=\sum_{f\in F1,F_2,F_3}f-\binom{k_1}{2}-\binom{k_2}{2}-\binom{k_3}{2}+1
$$
(which can be explicitly constructed using Theorem \ref{Teor1}).
\end{corollary}

\begin{proof}
If we write $\tilde F_j=\{f_{k_j}^{j\rceil}-f+1,f\in F_j\}$, $j=1,3$,  and $h_j=\begin{cases} \min F_j,& \mbox{if $F_j\not =\emptyset$,}\\
1,& \mbox{if $F_j\not =\emptyset$,}\end{cases}$ using (\ref{dinv}) and (\ref{dinv2}), one straightforwardly has $J_{h_j}(\tilde F_j)=I(F_j)$. Write now $\tilde \F=(\tilde F_1,F_2,\tilde F_3)$. It is now easy to see that
$$
\rho _{\alpha,\beta,N}^{\F}=\tilde \rho_{\alpha+f_{2,M}+f_{3,M}+2,\beta+f_{2,M}-f_{3,M},
N-f_{1,M}-f_{2,M}-2}^{\tilde \F,h_1,h_3}(x-f_{2,M}-1).
$$
The corollary is then a straightforward consequence of the previous Theorem.
\end{proof}

The hypothesis on $\Omega ^\U(n)\not =0$, for $0\le n\le N+m_1+m_2+1,$ in the previous Theorem and Corollary is then sufficient for the existence of a sequence of orthogonal polynomials with respect to the (possible signed) measure $\tilde \rho_{\alpha,\beta,N} ^{\F,h_1,h_3}$. We guess that this hypothesis is also necessary for the existence of such sequence of orthogonal polynomials.

\bigskip

Notice that there are different sets $F_1$ and $F_2$ for which the measures $\rho _{\alpha,\beta,N}^{\F}$
(\ref{udspm}) are equal. Each of these possibilities provides
a different representation for the orthogonal polynomials with respect to $\rho _{\alpha,\beta,N}^{\F}$ in the form (\ref{qusch}) and a different higher order difference operator with respect to which they are eigenfunctions. It is not difficult to see that only one of these possibilities satisfies the condition $f_{1,M},f_{2,M}<N/2$. This is the more interesting choice because it minimizes the order $2r$ of the associated higher order difference operator.
This fact will be clear with an example. Take $N=100$ and the measure $\mu=(x-1)(x-5)(x-68)\rho _{\alpha,\beta,N}$. There are eight couples of different sets $F_1$ and $F_2$ for which the measures $\mu$ and $\rho _{\alpha,\beta,N}^{\F}$ (\ref{udspm}) coincide (except for a sign). They are the following
\begin{align*}
F_1&=\{1,5,68\}, F_2=\emptyset, \quad &F_1&=\emptyset, F_2=\{31,94,98\},\\
F_1&=\{1,5\}, F_2=\{31\},\quad &F_1&=\{1,68\}, F_2=\{94\},\quad &F_1&=\{5,68\}, F_2=\{98\},\\
F_1&=\{1\}, F_2=\{31,94\},\quad &F_1&=\{5\}, F_2=\{31,98\},\quad &F_1&=\{68\}, F_2=\{94,98\}.
\end{align*}
Only one of these couples satisfies the assumption $f_{1,M},f_{2,M}<N/2$: $F_1=\{1,5\}$, $F_2=\{31\}$. Actually, it is easy to check that
 this couple
minimizes the number
$$
\sum_{f\in F1,F_2}f-\binom{k_1}{2}-\binom{k_2}{2}+1.
$$
Hence, it also minimizes de order $2r$ of the difference operator with respect to which the polynomials $(q_n)_n$ are eigenfunctions.

\section{Proofs of the Lemmas}\label{sproofs}
In this Section, we include the proofs of Lemmas \ref{lgp2}, \ref{lgp1} in Section 3 and Lemma \ref{dhdo} in Section 4.

\begin{proof}[Proof of the Lemma \ref{lgp2}]
To simplify the notation write $d_g=\deg (S(x)\Omega_g^h)$. By hypothesis, we have $d_g\le g+d_0$.
We can also write
\begin{equation}\label{spm0}
\Psi_j^h(x)=\sum_{i=0}^{\tilde d}a_i^{h,j}x^i.
\end{equation}
On the one hand, from the definition of $\Psi_j^h$, one has
\begin{equation}\label{spm1}
M_h(x)=\sum_{j=1}^m(-1)^{h+j}\Psi_j^h(x+j).
\end{equation}
On the other hand, by expanding the (quasi) Casorati determinant $\Omega _g^h$ by its $h$-row, we get
\begin{align*}
S(x)\Omega_g^h(x)&=S(x)\sum_{j=1}^m(-1)^{h+j}\xi_{x-j,m-j}^h(x-j+1)^g\det\left(\xi_{x-r,m-r}^{l}Y_l(x-r)\right)_{l\in \II_h;r\in \II_j}\\
&=\sum_{j=1}^m(-1)^{h+j}(x-j+1)^g\Psi_j^h(x).
\end{align*}
This shows that both $M_h$ and $S\Omega_g^h$ are polynomials in $x$. Moreover, the degree of $M_h$ is at most $\tilde d$. Hence, if $\tilde d\le d_0$, the proof is finished. We then assume that $\tilde d> d_0$.

Using that
$$
(x-j+1)^g=\sum_{v=0}^g(-1)^{g-v}\binom{g}{v}(x+1)^vj^{g-v},
$$
we get for $S\Omega_g^h$ the expansion
\begin{equation}\label{spm2}
S(x)\Omega_g^h(x)=(-1)^{h+g}\sum_{v=0}^g(-1)^{v}\binom{g}{v}(x+1)^v \sum_{j=1}^m(-1)^{j}j^{g-v}\Psi^h_j(x).
\end{equation}
We now prove by induction on $g$ that
\begin{equation}\label{spm3}
\sum_{j=1}^m(-1)^{j}j^ga_i^{h,j}=0,\quad \mbox{for $d_0+g<i\le \tilde d$.}
\end{equation}
Indeed, for $g=0$, the polynomial in the left hand side of (\ref{spm2}) has degree $d_0$, and the polynomial in the right hand side has degree at most $\tilde d$. The particular caso of (\ref{spm3}) for $g=0$ then follows from the expansion (\ref{spm0}). Assume now that (\ref{spm3}) holds for any nonnegative number $0,1,\cdots , g-1$.
Take now a number $i$ with $d_0+g<i\le \tilde d$. The induction hypothesis shows that for $v=1,\cdots , g$, then $\sum_{j=1}^m(-1)^{j}j^{g-v}a_i^{h,j}=0$. Hence the addends in the right hand side of (\ref{spm2}) corresponding to $v=1,\cdots , g$, have degree at most $d_0+g$. Since the polynomial in the left hand side of (\ref{spm2}) has degree at most $d_0+g$ as well, one can deduce that also the first addend ($v=0$) in the right hand side of (\ref{spm2}) has degree at most $d_0+g$. Using again (\ref{spm0}), we get that also $\sum_{j=1}^m(-1)^{j}j^{g}a_i^{h,j}=0$
for $d_0+g<i\le \tilde d$.

To finish the proof it is enough to  insert in (\ref{spm1}) the expansion of $\Psi_j^h$ and use (\ref{spm3}):
\begin{align*}
M_h(x)&=\sum_{j=1}^m(-1)^{h+j}\sum_{i=0}^{\tilde d}a_i^{h,j}(x+j)^i\\
&=(-1)^h\sum_{l=0}^{\tilde d}x^l\sum_{i=l}^{\tilde d}\binom{i}{l}\sum_{j=1}^m(-1)^{j}j^{i-l}a_i^{h,j}\\
&=(-1)^h\sum_{l=0}^{d_0}x^l\sum_{i=l}^{\tilde d}\binom{i}{l}\sum_{j=1}^m(-1)^{j}j^{i-l}a_i^{h,j}.
\end{align*}

\end{proof}

\begin{proof}[Proof of the Lemma \ref{lgp1}]
Consider the matrix
$$
Q(x)=\left(
  \begin{array}{@{}c@{}lccc@{}c@{}}
    &&&\hspace{-.9cm}{}_{1\le j\le m} \\
    \dosfilas{ Y_{i}(x-j) }{i\in \UU_1} \\
    \dosfilas{ s_{m-j}^{\beta+N+1}(x-j)s_{j-1}^{-\alpha+1}(x)Y_{i}(x-j) }{i\in \UU_2}\\
     \dosfilas{ s_{m-j}^{N+1}(x-j)s_{j-1}^{-\alpha+1}(x)Y_{i}(x-j) }{i\in \UU_3}
  \end{array}
  \hspace{-.4cm}\right),
$$
where $s_j^u(x)$ is the polynomials defined by (\ref{defsj}).

In order to prove that $P$ is a polynomial, it is enough to prove that
$$
x=\begin{cases} \beta+N+m-i,& i=0,\cdots, m_2-2,\\
N+m-i,& i=0,\cdots, m_3-2,\\
-\alpha+i+1,& i=0,\cdots, m_2+m_3-2,
\end{cases}
$$
are root of $\det Q(x)$ of multiplicities at least $m_2-i-1$, $m_3-i-1$ and $m_2+m_3-i-1$, respectively. It is easy to see that $\rank Q(\beta+N+m-i)\le m_1+m_3+i+1$. Hence $0$ is an eigenvalue of $Q(\beta+N+m-i)$ of geometric multiplicity at least $m-m_1-m_3-i-1=m_2-i-1$, and so $0$ is an eigenvalue of $Q(\beta+N+m-i)$ of algebraic multiplicity at least $m_2-i-1$. This implies that $x=\beta+N+m-i$ is a root of $\det Q(x)$ of multiplicity at least $m_2-i-1$. For the other values of $x$ the proof is similar.

The lemma follows now easily from the following claim: \textit{$\det Q(x)$ is a polynomial of degree $d+m_2^2+m_3^2+(m_2-1)(m_3-1)-1$, with leading coefficient given by $(-1)^{m_2^2+m_3^2+(m_2-1)(m_3-1)-1} r$, where $d$ and $r$ are defined by (\ref{ddp}) and (\ref{mspcl})}, respectively.

We now prove the claim.

Write $\phi_{\U}$ for the special case of the determinant $\det Q(x)$ when
\begin{equation}\label{mpri}
Y_i(x)=\begin{cases} (x+1)_{u_i^{1\rceil}},& i\in \UU_1,\\
(\beta+N-u_i^{2\rceil}-x+1)_{u_i^{2\rceil}},& i\in \UU_2,\\
(N-u_i^{3\rceil}-x+1)_{u_i^{3\rceil}},& i\in \UU_3.
\end{cases}
\end{equation}
We first prove that the  claim follows if we prove it
for $\phi_{\U}$.
Indeed, for arbitrary polynomials $Y_i$ with leading coefficient equal to $r _i$, we can write
$$
\det Q(x)=\left(\prod_{i=1}^kr_i\right)\phi_{\U}+\sum a_{\V}\phi_{\V},
$$
where the sum is taken over all the trios $\V=(V_1,V_2,V_3)$ satisfying that $0\le v_i^{j\rceil}\le u_i^{j\rceil}$, for $i \in \UU_j$, $j=1,2,3$, and at least for some $j$ and $i_0$ with $i_0\in \UU_j$,
$v_{i_0}^{j\rceil}< u_{i_0}^{j\rceil}$. The claim for $\det Q$ now follows easily.

We finally prove the claim for $\phi_{\U}$. From the definition of $Y_i$, we get
\begin{equation}\label{dpc}
\phi_{\U}(x)=\left|\begin{array}{@{}c@{}lccc@{}c@{}}
    &&&\hspace{-.9cm}{}_{1\le j\le m} \\
    \dosfilas{(x-j+1)_u }{u\in U_1} \\
    \dosfilas{s_{m-j}^{\beta+N+1}(x-j)s_{j-1}^{-\alpha+1}(x)s_{u}^{\beta+N+1-u}(x-j) }{u\in U_2}\\
     \dosfilas{ s_{m-j}^{N+1}(x-j)s_{j-1}^{-\alpha+1}(x)s_{u}^{N+1-u}(x-j) }{u\in U_3}
  \end{array}
  \hspace{-.4cm}\right|.
\end{equation}
Taking into account that
\begin{align*}
s_{m-j}^{\beta+N+1}(x-j)s_{u}^{\beta+N+1-u}(x-j) &=
s_{u}^{\beta+N+m-u+1}(x)s_{m-j}^{\beta+N-u+1}(x-j),\\
s_{m-j}^{N+1}(x-j)s_{u}^{N+1-u}(x-j) &=
s_{u}^{N+m-u+1}(x)s_{m-j}^{N-u+1}(x-j),
\end{align*}
we can rewrite (\ref{dpc})
\begin{equation}\label{dpc-}
\phi_{\U}(x)=f(x)\left|\begin{array}{@{}c@{}lccc@{}c@{}}
    &&&\hspace{-.9cm}{}_{1\le j\le m} \\
    \dosfilas{(x-j+1)_u }{u\in U_1} \\
    \dosfilas{s_{m-j}^{\beta+N-u+1}(x-j)s_{j-1}^{-\alpha+1}(x) }{u\in U_2}\\
     \dosfilas{s_{m-j}^{N-u+1}(x-j)s_{j-1}^{-\alpha+1}(x)}{u\in U_3}
  \end{array}
  \hspace{-.4cm}\right|,
\end{equation}
where $f(x)=\prod_{u\in U_2}s_{u}^{\beta+N+m-u+1}(x)\prod_{u\in U_3}s_{u}^{N+m-u+1}(x)$.
Subtracting columns in (\ref{dpc-}), we get
$$
\phi_{\U}(x)=(-1)^{\binom{m}{2}}f(x)\left|\begin{array}{@{}c@{}lccc@{}c@{}}
    &&&\hspace{-.9cm}{}_{1\le j\le m} \\
    \dosfilas{(u-j+2)_{j-1}(x)_{u-j+1} }{u\in U_1} \\
    \dosfilas{(\alpha+\beta+N-u+1)_{j-1}s_{m-j}^{\beta+N-u+1}(x-j) }{u\in U_2}\\
     \dosfilas{(\alpha+N-u+1)_{j-1}s_{m-j}^{N-u+1}(x-j)}{u\in U_3}
  \end{array}
  \hspace{-.4cm}\right|.
$$
From where it is easy to see that $\phi_{\U}$ is a polynomial of degree at most $\tilde d=d+m_2^2+m_3^2+(m_2-1)(m_3-1)-1$ with  coefficient of the power
$x^{\tilde d}$ equal to
$$
(-1)^{\binom{m}{2}+\sum_{u\in U_2,U_3}u}
 \left|\begin{array}{@{}c@{}lccc@{}c@{}}
    &&&\hspace{-.9cm}{}_{1\le j\le m} \\
    \dosfilas{(u-j+2)_{j-1} }{u\in U_1} \\
    \dosfilas{(-1)^{m-j}(\alpha+\beta+N-u+1)_{j-1} }{u\in U_2}\\
     \dosfilas{(-1)^{m-j}(\alpha+N-u+1)_{j-1}}{u\in U_3}
  \end{array}
  \hspace{-.4cm}\right|.
$$
The proof can be concluded by computing this determinant using standard determinantal techniques.
\end{proof}

\begin{proof}[Proof of the Lemma \ref{dhdo}]
We prove the Lemma only for $\mathcal D_1$ (the proof for the other $\D$-operators is similar and it is omitted).

Firs of all, we show that $\mathcal D_1\in \A^\lambda$.  We have to prove that for any polynomial $p(\lambda)$ we have $\D_1(p)\in \PP^\lambda$.
Using the characterization of $\PP^\lambda$ given in (\ref{definvsr}), it is enough to show that
$\I (\D_1(p))=\D_1(p)$, where $\I$ is the involution defined by (\ref{definvsr}). But this follows straightforwardly, taking into account that $\I (\Delta_x(p))=-\nabla_x(p)$ and
$$
\I \left(\frac{(x+\alpha+1)(x+\alpha+\beta+1)}{(2x+\alpha+\beta+1)(2x+\alpha+\beta+2)}\right)=\frac{x(x+\beta)}{(2x+\alpha+\beta)(2x+\alpha+\beta+1)}.
$$
Using the forward shift operator (\ref{eps2}) for the dual Hahn polynomials we get
\begin{align}\label{pss}
\D _1(R_n(\lambda(x))&=\frac{-(x+\alpha+1)(x+\alpha+\beta+1)R_{n-1}^{\alpha+1,\beta,N-1}(\lambda^{\alpha+1,\beta}(x))}
{(\alpha+1)(2x+\alpha+\beta+1)}\\
\nonumber &\hspace{2cm} +\frac{x(x+\beta)R_{n-1}^{\alpha+1,\beta,N-1}(\lambda^{\alpha+1,\beta}(x-1))}
{(\alpha+1)(2x+\alpha+\beta+1)}.
\end{align}
Using (\ref{defTo}), the definition of $\varepsilon_{n,1}$ (\ref{eps1}) and (\ref{dhmp1}), we have that (\ref{pss}) is equivalent to
\begin{align*}
&\frac{-(x+\alpha+1)(x+\alpha+\beta+1)R_{n-1}^{\alpha+1,\beta,N-1}(\lambda^{\alpha+1,\beta}(x))}
{(\alpha+1)(2x+\alpha+\beta+1)}
\\ &\hspace{2cm}+\frac{x(x+\beta)R_{n-1}^{\alpha+1,\beta,N-1}(\lambda^{\alpha+1,\beta}(x-1))}{(\alpha+1)(2x+\alpha+\beta+1)}=
R_n(\lambda(x))-R_n^{\alpha,\beta,N-1}(\lambda(x)).
\end{align*}
Notice that since the left and right hand terms in this identity are polynomials in $N$ and rational functions in $x$, it will be enough to prove it for $x\in \NN$ and $N\not \in \NN$. Hence, taking into account the generating function (\ref{dhmp3}) for the dual Hahn polynomials, this identity is equivalent to
\begin{align*}
&-\frac{(x+\alpha+1)(x+\alpha+\beta+1)\pFq{2}{1}{-x,-x-\beta}{\alpha+2}{t}}{(\alpha+1)(2x+\alpha+\beta+1)}
\\
&\hspace{2cm} +\frac{x(x+\beta)(1-t)\pFq{2}{1}{-x+1,-x+1-\beta}{\alpha+2}{t}}{(\alpha+1)(2x+\alpha+\beta+1)}
=-\pFq{2}{1}{-x,-x-\beta}{\alpha+1}{t}.
\end{align*}
But this last identity can be checked easily from the power expansion of the hypergeometric function.
\end{proof}

\end{document}